\documentclass{amsart}
\usepackage[utf8]{inputenc}

\usepackage{orcidlink}

\usepackage{comment}
\usepackage{amsmath,amsfonts,amssymb,amsthm}
\usepackage{hyperref}
\usepackage{thmtools}
\usepackage{xcolor}
\usepackage{tikz,tikz-cd}
\usetikzlibrary{matrix,arrows}
\usepackage{enumitem}
\usepackage{mathabx}
\usepackage{multirow}
\usepackage{float}
\usepackage[capitalise]{cleveref}
\usepackage{verbatim}

\newcommand{\id}{\operatorname{id}}
\newcommand{\Hom}{\operatorname{Hom}}
\newcommand{\End}{\operatorname{End}}
\newcommand{\Aut}{\operatorname{Aut}}

\newcommand{\Q}{\mathbb{Q}}
\newcommand{\Z}{\mathbb{Z}}
\newcommand{\N}{\mathbb{N}}
\newcommand{\Ann}{\operatorname{Ann}}
\newcommand{\Soc}{\operatorname{Soc}}
\newcommand{\Fix}{\operatorname{Fix}}

\newcommand{\Stab}{\operatorname{Stab}}
\newcommand{\Triv}{\operatorname{Triv}}
\newcommand{\op}{\operatorname{op}}

\newcommand{\CD}{\operatorname{CD}}
\newcommand{\PStab}{\operatorname{PStab}}

\makeatletter
\newtheorem{thm}{Theorem}[section]
\newtheorem{lem}[thm]{Lemma}
\newtheorem{cor}[thm]{Corollary}
\newtheorem{pro}[thm]{Proposition}

\newtheorem{defn}[thm]{Definition}

\newtheorem{question}[thm]{Question}

\newtheorem{rem}[thm]{Remark}
\newtheorem{exa}[thm]{Example}

\makeatother

\usepackage{ifthen, xcolor}
\newboolean{commentBoolVar}
\setboolean{commentBoolVar}{true} 
\newcommand{\colourcomment}[3]
{%
\ifthenelse{\boolean{commentBoolVar}}{{\color{#2}(#1: #3)}}{}%
}%

\title[Finiteness conditions on skew braces and solutions of YBE]{Finiteness conditions on skew braces and solutions of the Yang-Baxter equation}

\author[Rosa Cascella]{Rosa Cascella\orcidlink{0009-0006-6386-7252}}

\address{%
Dipartimento di Matematica e Applicazioni "Renato Caccioppoli", \\
Università degli Studi di Napoli Federico II,\\
Via Cinthia 26,\\
80126 Naples,\\
Italy}

\email{rosa.cascella@unina.it}
\author[Silvia Properzi]{Silvia Properzi\orcidlink{0000-0002-8228-9156}}
\address{Department of Mathematics and Data Science,\\
Vrije Universiteit Brussel,\\
Pleinlaan 2,\\
1050 Brussels,\\
Belgium}
\email{silvia.properzi@vub.be}

\author[Arne Van Antwerpen]{Arne Van Antwerpen\orcidlink{0000-0001-7619-6298}}
\address{Department of Mathematics and Statistics,\\
Maynooth University,\\
Maynooth,\\
Ireland}
\email{Arne.VanAntwerpen@mu.ie}

\subjclass{Primary 16T25; Secondary 81R50, 20F24, 20N99}

\keywords{Yang-Baxter equation, skew brace, FC-group, finiteness condition}

\date{\today}

\begin{document}

\begin{abstract}
  A finite non-degenerate set-theoretic solution $(X,r)$ of the Yang-Baxter equation gives rise to a structure skew brace $B(X,r)$
  that is a $\lambda_f$-skew brace, i.e.
  every element has finitely many $\lambda$-images, and whose additive group is $FC$.
  This motivates the study of finiteness conditions on skew braces.
  We first study the general class of $\lambda_f$ skew braces and the subclass where the additive group is $FC$, showing that these properties share a resemblance to finite conjugacy, having an analog of the $FC$-center and several analogous structural results.
 Furthermore, by passing through the structure skew brace of a solution,
 this property measures whether elements are contained in a finite decomposition factor,
 identifying a class of infinite solutions that may exhibit similar properties to finite ones.
 Finally, we show that for a sub skew brace where both groups have finite index,
 both indices need to coincide and that such a sub skew brace contains a strong left ideal of finite index.
\end{abstract}

\maketitle

\section{Introduction}

The Yang-Baxter equation (YBE) is a fundamental equation originating in mathematical physics in the work of Nobel Laureate C.N. Yang \cite{MR261870} and R. Baxter \cite{MR690578}, which is intimately connected to diverse fields of mathematics, such as TQFT's, quantum computation, representations of the braid group, knot theory, Hopf algebras and quantum groups. Recall that a solution to the Yang-Baxter equation is a tuple $(V,R)$ where $V$ is a vector space and $R\colon V \otimes V \rightarrow V \otimes V$ is a bijective linear map such that on $V^{\otimes 3}$ one has that $$(R \otimes \id_V) (\id_V\otimes R)(R \otimes \id_V) = (\id_V\otimes R)(R \otimes \id_V) (\id_V\otimes R).$$ 
In \cite{MR1183474} Drinfel'd remarked that many families of solutions, known at that time, stem from so-called set-theoretic solutions. A set-theoretic solution of the Yang-Baxter equation is a tuple $(X,r)$ where $X$ is a non-empty set and $r:X \times X \rightarrow X \times X$ is a bijective map satisfying on $X^3$ the equality $$ (r \times \id_X)(\id_X \times r)(r \times \id_X) = (\id_X \times r)(r \times \id_X)(\id_X \times r).$$ 

Due to the vast number of set-theoretic solutions of the YBE, attention was mostly restricted to bijective non-degenerate solutions, which we will call solutions henceforth. An approach to solutions through algebraic structures was initiated by seminal works of Etingof, Schedler and Soloviev \cite{MR1722951} and Gateva-Ivanova and Van den Bergh \cite{MR1637256}.

Recently, Rump \cite{MR2278047} and Guarnieri and Vendramin \cite{MR3647970} introduced skew braces, algebraic structures that are deeply linked to bijective non-degenerate set-theoretic solutions of the Yang-Baxter equation. A skew brace is a set $B$ with two group operations $(B,+)$ and $(B,\circ)$ that satisfy a skew left distributivity condition. This skew left distributivity gives rise to an action $\lambda$ by automorphisms of $(B,\circ)$ on $(B,+)$, which is at the core of the structure of the skew brace. 
Moreover, utilizing the $\lambda$-map, it was shown \cite{MR3647970, LYZset,MR2278047} that a skew brace has an associated solution and vice versa every solution has an associated skew brace. This uncovered a deep link between algebraic properties of skew braces and combinatorial properties of set-theoretic solutions, which has been extensively explored (cf. \cite{MR3815290,Cast22, MR4457900, MR3957824,s, MR4470170,MR4023387, cablingorig, MR4163866}). For example, in \cite{MR4023387} it was shown that a decomposition of a solution into a disjoint union of solutions translates on skew braces to a factorization into so-called strong left ideals. Recently, interest has grown in the study of skew braces and solutions of infinite order \cite{MR4718182, index, conju, dedekind}. In this article, we introduce a class of infinite solutions that behave similarly to finite solutions, namely solutions where every element is contained in a finite decomposition factor.

It turns out that this can be translated to a condition on skew braces, which is reminiscent of so-called Finite Conjugacy ($FC$)-groups in group theory, originally introduced by R. Baer \cite{FC}. The behavior of $FC$-groups has been investigated by several authors over the last seventy years (cf. the monograph \cite{Tomkinson} and references therein); among the basic properties of $FC$-groups, it is well known that the commutator subgroup of any $FC$-group is locally finite \cite{neumann} and that finitely generated groups with the $FC$ property are central-by-finite \cite{FC}. 

In a skew brace there is a plethora of choices for a notion of conjugate of an element, as remarked in \cite{conju}. Indeed, one has the conjugates of the additive group $(B,+)$, the multiplicative group $(B,\circ)$ and $\lambda$-images of an element. In \cite{s}, so-called $(S)$-skew braces are introduced, i.e. skew braces where every element has only finitely many $\lambda$-images and finitely many additive and multiplicative conjugates. 
Among structural properties, which are similar to those of $FC$-groups, it turns out that finite solutions whose associated skew brace has this property are necessarily derived indecomposable. The notion of an $(S)$-skew brace corresponds to demanding that the semi-direct product $(B,+)\rtimes_{\lambda}(B,\circ)$ is an $FC$-group. 

In this paper, we focus on elements with a finite amount of additive conjugates and a finite amount of $\lambda$-images, which we call $\theta_f$-elements and skew braces where all elements are $\theta_f$. The latter is equivalent to asking that the natural action of the semi-direct product $(B,+)\rtimes_\lambda (B,\circ)$ has finite orbits. This corresponds to solutions where every element is contained in a finite decomposition factor and hence covers all finite solutions.

In the first section, we provide preliminaries to make the text self-contained and several examples of infinite skew braces, which we use throughout the text as examples to illustrate the sharpness of several results. 
In \cref{secindex}, we discuss the notion of index of a sub skew brace. 
In \cite{index} it was remarked that a natural extension of this notion to the index of a sub skew brace $A$ in a skew brace $B$ is unclear.
A priori, the index $|(B,+):(A,+)|$  may differ from the $|(B,\circ):(A,\circ)|$. One says that a skew brace $B$ is index-preserving if for any sub skew brace $A$ one has that if $|(B,+):(A,+)|$ or $|(B,\circ):(A,\circ)|$ is finite, then both are finite and equal. 
In \cite{index,dedekind} several closure properties of the class of index-preserving skew braces were shown. 
Entailing, for example, that all weakly soluble skew braces are index-preserving.
We contribute to this study by showing that if $|(B,+):(A,+)|$ and $|(B,\circ):(A,\circ)|$ are both finite, then these indices are necessarily equal. 
The proof reveals a result of independent interest, namely every sub skew brace of finite index contains a strong left ideal of finite index, which extends a similar group-theoretic result.
We show that for two-sided skew braces this can even be strengthened to an ideal of finite index. 
In \cref{seclambda}, we examine $\lambda_f$-elements, i.e. those that only have a finite amount of $\lambda$-images, and skew braces where all elements are $\lambda_f$. We show that the notion of being finitely generated is well-defined and show some structural results. 
In \cref{sectheta} we study $\lambda_f$-skew braces where the additive group is $FC$, called $\theta_f$.
These have a structure similar to $FC$-groups, where the socle plays an analogous role to the center in group theory.
In particular, we show an analog of Dietzmann's lemma and that finitely generated $\theta_f$ skew braces have a socle of finite index.
Moreover, we establish a characterization similar to \cite[Theorem 2]{FC}. 
Finally, in section \cref{secsolution} we connect this novel algebraic notion to solutions. We show that the structure skew brace of a solution $(X,r)$ is $\theta_f$ if and only if in the injectivization of $(X,r)$ every element is contained in a finite decomposition factor. As the structure skew brace of a finite solution is $\theta_f$, this identifies a class of solutions whose behavior is similar to that of finite solutions.

\section{Preliminaries on skew braces and solutions}

    A \textit{skew brace} is a triple $(B,+,\circ)$, where $(B,+)$ and $(B,\circ)$ are groups such that for all $a,b,c \in B$ one has that $$ a \circ (b+c) = a \circ b - a + a\circ c,$$
    where $-a$ denotes the inverse of $a$ in $(B,+)$ and for ease of notation $\circ$ takes precedence over $+$. 
    Similarly, we denote $\overline{a}$ the inverse of $a$ in $(B,\circ)$.
    If $B$ furthermore satisfies the right skew distributivity law $$ (c+b) \circ a = c\circ a -a + b\circ a,$$ for all $a,b,c \in B$, then $B$ is said to be a \textit{two-sided skew brace}.

    For a skew brace $(B,+,\circ)$, the group $(B,+)$ is called the \textit{additive group} of $B$ and $(B,\circ)$ is called the \textit{multiplicative group} of $B$. 
    Let $\chi$ be a group theoretical property (e.g. abelian, nilpotent, ...), then one says that $B$ is of $\chi$-type if $(B,+)$ is a $\chi$-group.

    One says that a skew brace $B$ is \textit{trivial}, if the operations coincide, i.e. for all $a,b \in B$ one has that $a\circ b =  a+b$. Similarly, a skew brace $B$ is called \textit{almost trivial}, if for all $a,b\in B$ one has that $a\circ b= b+a$. Moreover, a skew brace is called \textit{weakly trivial}, if it is a sub skew brace of the direct product of a trivial and an almost trivial skew brace.

    Given a skew brace $(B,+,\circ)$, one has that $(B,+^{op},\circ)$ is a skew brace, which is called the \textit{opposite skew brace}.
    
    Central to the theory of skew braces is the so-called \textit{$\lambda$-morphism}, which measures the difference between both operations. Denote for any $a,b \in B$ the expression $$ \lambda_a(b) = -a+a\circ b.$$
    Then, it was shown that $\lambda_a \in \Aut(B,+)$. Moreover, the map $\lambda\colon (B,\circ) \rightarrow \Aut(B,+)$ is a group morphism. The $\lambda$-map provides a clear connection between the additive and multiplicative structure of a skew brace, which is condensed into the equality $$ a\circ b =  a + \lambda_a(b).$$
    A \textit{sub skew brace} of a skew brace $B$ is a subgroup $A$ of $(B,+)$ such that $\lambda_a(A)=A$ for all $a\in A$. This coincides with the fact that $A$ is also a subgroup of $(B,\circ)$.
     We will denote with $|B:A|_+$ the index of $(A,+)$ in $(B,+)$ and
     with $|B:A|_\circ$ the index of $(A,\circ)$ in $(B,\circ)$.
    A \textit{left ideal} of $B$ is a subgroup of $(B,+)$ such that for all $b \in B$ one has that $\lambda_b(L)=L$. 
    Moreover, if $L$ is a left ideal of $B$ and normal in $(B,+)$, then $L$ is called a \textit{strong left ideal}. 
    Finally, an \textit{ideal} of $B$ is a strong left ideal $I$ of $B$ such that $I$ is normal in $(B,\circ)$. The latter, as expected, coincides with the kernel of a \textit{skew brace morphism}, i.e. a map between two skew braces that is a group morphism in both the additive and multiplicative groups.

 Let $B$ be a skew brace. Then
 $\lambda$ is an action by automorphisms of $(B,\circ)$ on $(B,+)$ 
and there is an action by automorphisms of $(B,+)\rtimes_{\lambda} (B,\circ)$ on $(B,+)$ defined by 
\[
\theta_{(a,b)}(c)=a+\lambda_b(c)-a.
\]

Let $B$ be a skew brace and $x\in B$. We denote the $\lambda$-\emph{stabilizer of $x$ in $B$} by $\Stab_\lambda(x)=\{a\in B\colon \lambda_a(x)=x\}$.  It is easy to observe that $\Stab_\lambda(x)\leq (B,\circ)$. 
The \emph{$\lambda$-orbit of $x$ in $B$} is $[x]_\lambda=\{\lambda_a(x)\colon a\in B\}$.
Moreover, the set of elements with a trivial $\lambda$-orbit $\Fix(B)=\{a\in B\colon \lambda_x(a)=a \textnormal{ for all } x\in B\}$ is a trivial sub skew brace of $B$.
By the Orbit-Stabilizer theorem, $\left|[a]_\lambda\right|\left|\Stab_\lambda(a)\right|=|B|$ for every $a\in B$.   

 Similarly, the \emph{$\theta$-stabilizer of $x$ in $B$} 
is 
\[
\Stab_\theta(x)=\{(a,b)\in (B,+)\rtimes_{\lambda} (B,\circ)\colon \theta_{(a,b)}(x)=x\}.
\]

The \emph{$\theta$-orbit of $x$ in $B$} is $[x]_\theta=\{\theta_{(a,b)}(x) \colon (a,b)\in (B,+)\rtimes_{\lambda} (B,\circ)\}$.
By the Orbit-Stabilizer theorem, $\left|[a]_\theta\right|\left|\Stab_\theta(a)\right|=|B|^2$ for every $a\in B$.

The \emph{socle of $B$} is the ideal $\Soc(B)=\ker\lambda\cap Z(B,+)$ and 
the \emph{annihilator of $B$} is the ideal $\Ann(B)=\Soc(B)\cap Z(B,\circ)$.

\begin{defn}
Let $B$ be a skew brace and $H\leq (B,+)$. We denote for every $b\in B$,
\[
\lambda_b(H)=\{\lambda_b(h)\mid h\in H\}.
\]
Moreover, we set 
\[
\Stab_\lambda(H)=\{b\in B\mid \lambda_b(H)=H\}
\]
and
\[
\PStab_B(H)=\{b\in B\mid \lambda_b(h)=h\text{ for every }h\in H\}
\]   

\end{defn}

\begin{pro}\label{quotient}
Let $B$ be a skew brace and $H$ be a subgroup of $(B,+)$.
Then $\PStab_B(H)$ is a normal subgroup of $(\Stab_\lambda(H),\circ)$ with $(\Stab_\lambda(H)/\PStab_B(H),\circ)$ isomorphic to a subgroup of $\Aut(H,+)$.
\end{pro}

\begin{proof}
    Let us consider $b\in \Stab_\lambda(H)$, $x\in \PStab_B(H)$ and $h\in H$.
    Then 
    \[
    \lambda_{b\circ x\circ \overline{b}}(h)=\lambda_{b}(\lambda_x(\lambda_{\overline{b}}(h)))=\lambda_{b}(\lambda_{\overline{b}}(h))=h
    \]
    and $\PStab_B(H)\lhd (\Stab_\lambda(H),\circ)$. Observe that if $b$ is an element of $\Stab_\lambda(H)$, the function $\lambda_b\big|_H$ is an automorphism of $H$. Finally, the function $$\tau:b\in \Stab_\lambda(H)\to \lambda_b\big|_H\in \Aut(H,+)$$ is a homomorphism, where $\ker \tau=\PStab_B(H)$. 
\end{proof}

For all $a,b\in B$ we denote $a*b=\lambda_a(b)-b$.
Let $X,Y$ be subsets of $A$. Then $X*Y$ is the additive subgroup of $B$ generated by $x*y$ for all $x\in X$ and $y\in Y$.
We set $B^{2}=B*B$ and $B^{2}_{op} = B \ast_{op}B$, where $\ast_{op}$ corresponds to the star operation in the opposite skew brace of $B$.

The \emph{commutator} $B'$ of $B$ is the additive
subgroup of $B$ generated by the commutator $[B, B]_+$ of $(B,+)$ and $B^2$. 
For all $x,a\in B$ we denote 
$[x]_\circ=\{a\circ x \circ \overline{a}\mid a\in B\}$
and
$[x]_+=\{a+x-a\mid a\in B\}$. 

Now we note some elementary properties necessary in the subsequent sections. First, we recall the characterization of two-sided skew braces of abelian type \cite{Lau, MR2278047}.

\begin{pro}
    A skew brace $(B,+,\circ)$ is a two-sided skew brace of abelian type if and only if $(B,+,*)$ is a radical ring. Moreover, the notion of (left) ideals of both structures coincides.
\end{pro}

Moreover, two-sided skew braces have a desirable property, which breaks down for most other skew braces.

\begin{rem}
\label{mult conj in two-sided}
    If $B$ is a two-sided skew brace, then 
    multiplicative conjugation is a skew brace automorphism.
\end{rem}

In \cite{dedekind} Caranti, Del Corso, Ferrara and Trombetti established that the class of skew braces for which the index is well-defined is local and closed under extensions. 

\begin{pro}\label{indexdedekind}
Let $B$ be a skew brace 
that has an ideal $I$ 
such that for every sub skew brace $D$ of $I$ and sub skew brace $C$ of $B/I$
    \[
    |I:D|_+=|I:D|_\circ \text{ and } |B/I:C|_+=|B/I:C|_\circ.
    \]

Then for every sub skew brace $A$ of $B$ 
    \[
    |B:A|_+=|B:A|_\circ. 
    \]
\end{pro}

In the remainder of this section we recall the link between solutions of the Yang-Baxter equation and skew braces and establish some examples of skew braces that we will refer to throughout the text.

 \begin{defn}
     A set-theoretic solution to the Yang-Baxter equation is a tuple $(X,r)$ such that $X$ is a non-empty set and $r\colon X \times X \rightarrow X \times X$ is a bijective map such that $$ (r \times \id_X)(\id_X \times r)( r \times \id_X) = (\id_X \times r)(r \times \id_X)(\id_X\times r).$$  Denote $r(x,y) = (\lambda_x(y),\rho_y(x))$. If $\lambda_x,\rho_x$ are bijective for all $x \in X$, then one says that $(X,r)$ is non-degenerate. 
 \end{defn}
Note that it was recently proven in \cite{MR5009092} that
the bijectivity of $r$ follows from the non-degeneracy condition.

From now on we abbreviate non-degenerate set-theoretic solution by {\it solution}. Skew braces both generate and govern solutions in the following sense. Every skew brace $B$ gives rise to a solution $(B,r_B)$, where 
\[
r_B(a,b) = ( \lambda_a(b), \overline{\lambda_a(b)}\circ a \circ b).
\]
Similarly, for a solution $(X,r)$ the structure group $G(X,r)$ is the group
$$ \left< x \in X \mid x\circ y = \lambda_x(y) \circ \rho_y(x)\right>.$$ 
Then, $G(X,r)$ has a natural skew brace structure $B(X,r)$, 
where the additive group is isomorphic to $G(X,r')$, 
where $(X,r')$ is the derived solution of $(X,r)$ 
defined by $r'(x,y) = (y, \lambda_{y}\rho_{\lambda_x^{-1}(y)}(x))$.
If $(X,r)$ is an \emph{involutive solution}, i.e. $r^2=\id_{X^2}$,
then $B(X,r)$ is a skew brace of abelian type. 
Furthermore, for involutive solutions, the canonical mapping $\iota :X \rightarrow B(X,r)$ is injective.
Solutions for which $\iota$ is injective are called \textit{injective solutions}.
One can consider the image $\iota(X)$ in $B(X,r)$, which is a subsolution of $(B(X,r),r_B)$
and has the property that 
$$ B(X,r) \cong B(\iota(X),r_B).$$
One says that $(\iota(X),r_B)$ is the \emph{injectivization of $(X,r)$}
and there is a canonical epimorphism of $X$ to its injectivization, induced by $\iota$.

\begin{defn}
    Let $(X,r)$ be a solution. Then, one defines the \emph{retract relation on $X$}
    as the relation $x \sim y$ if $\lambda_x=\lambda_y$ and $\rho_x=\rho_y$. 
    Then, $r$ induces a solution $\textup{Ret}(X,r)$ on $X/\sim$, which is called the \emph{retract of $(X,r)$}. Moreover, there is a canonical epimorphism $(X,r) \rightarrow \textup{Ret}(X,r)$.
\end{defn}
In \cite{jespers2019structure} it was shown, by construction of a cancellative congruence, 
that the injectivization morphism is a factor of the retract morphism. 
In \cite{MR1722951} the notion of a decomposable and indecomposable solution was introduced.
The latter should be viewed as building blocks that can be linked together. 
Below, we give a slightly different twist on this notion,
as we will be explicitly interested in the decomposition itself.

\begin{defn}
    Let $(X,r)$ be a solution. If there exists a partition $X=Y \cup Z$ such that $ (Y,r_{Y\times Y})$ and $(Z, r_{Z \times Z})$ are subsolutions, then the sets $Y$ and $Z$ are called \emph{decomposition factors}.
\end{defn}

Note that both the union and intersection of decomposition factors are decomposition factors.
As the complement of a decomposition factor is a decomposition factor by definition,
it follows that they form a sub Boolean algebra of the power set of $X$. 
A decomposition of a solution $(X,r)$ gives rise to a factorization of a skew brace as the sum of two strong left ideals.

\begin{thm}\label{decompgivesfactorization}
    Let $(X,r)$ be a solution and $Y$ a decomposition factor. Denote $Z=X \setminus Y$. Then, $B(X,r)$ is the sum of the strong left ideals $\left<Y\right>_+$ and $\left<Z\right>_+$.
\end{thm}
For a solution $(X,r)$ and a subsolution $(Y,r_{Y \times Y})$ one has that the sub skew brace $\left< Y \right>_+$ of $B(X,r)$ is a homomorphic image of $B(Y,r_{Y\times Y})$. Hence, \cref{decompgivesfactorization} gives a method to relate the structure skew brace on $X$ with the structure skew braces of decomposition factors.

We continue this section with some examples that will be referred to throughout the text.

\begin{exa}\label{exaoptriv}
Let $(G,\cdot)$ be a group. Then, the $\lambda$-action of $B=\op\Triv(G)$ coincides with conjugation in $G$:
\[
\lambda_a(x)=-a+a\circ x=a^{-1}\cdot (a\cdot_{\op} x)=a^{-1} \cdot x\cdot a,
\]
for every $a,x\in B$.

Moreover $a*b=\lambda_a(b)-b=a^{-1} \cdot b\cdot a\cdot b^{-1}=[a,b^{-1}]_\cdot$ for every $a,b\in B$, so
 $\op\Triv(G)^2$ coincides with the derived subgroup of $G$.
 
The $\theta$-action of $B=\op\Triv(G)$ also coincides with conjugation, as
$\theta_{(a,b)}$ is the conjugation in $G$ by $a\cdot b^{-1}$:
\[
\theta_{(a,b)}(x)=a-b+b\circ x-a=a\cdot b^{-1}\cdot (b\cdot_{\op}x)\cdot a^{-1}=a\cdot b^{-1}\cdot x\cdot b\cdot a^{-1},
\]
for every $a,b,x\in B$.

Hence, sub skew braces of $\op\Triv(G)$ are subgroups of $G$, left ideals of $\op\Triv(G)$ coincide with ideals and with normal subgroups of $G$.
\end{exa}

Recall that one can define a semidirect product of skew braces as follows, a more general version is known in \cite{Facchini}.

\begin{defn}
    Let $(A,\cdot,\cdot)$ be a trivial skew brace and $(B,+,\circ)$ a skew brace. Then, for any group morphism $\varphi: (B,\circ) \rightarrow \Aut(A,\cdot)$ one defines a skew brace
    with additive group $(A,\cdot) \times (B,+)$ and multiplicative group $(A,\cdot) \rtimes_{\varphi} (B,\circ)$, which is called the semidirect product of $A$ and $B$ and denoted $A \rtimes_{\varphi} B$.
\end{defn}

\begin{exa}\label{rosita}
Denote by $B$ the skew brace obtained as the semi-direct product of the trivial skew braces on $\mathbb{Z}_3 \times \mathbb{Q}$ and $\Z$ via the morphism \begin{align*}
    \varphi:\Z&\to \Aut(\Z/3\times\Q)\\
    k&\mapsto \varphi_k:(a,x)\mapsto \big((-1)^ka,2^kx\big).
\end{align*}
Then,  $(B,+) \cong \Z/3\times \Q\times \Z$ and $(B,\circ)$ is a group isomorphic to $(\Z_3\times \Q)\rtimes_\varphi \Z$, where
\[
(a,x,k)\circ(b,y,l) = \big(a+(-1)^kb, x+2^ky, k+l\big),
\]
for all $(a,x,k),(b,y,l)\in B$.

Moreover, 
$$\lambda_{(a,x,k)}(b,y,l)=\big((-1)^kb, 2^ky, l\big),$$
$$\overline{(b,y,l)}=((-1)^{l+1}b,-2^ly,-l).$$

\end{exa}
Finally, we construct an example of a skew brace of cyclic type where the multiplicative group is not $FC$.
\begin{exa}\label{es1}

     We can define on $\Z$ the following operation:
    \[
    n\circ m=\begin{cases}
        n+m, &\text{ if }n \text{ is even }\\
        n-m &\text{ if }n \text{ is odd }
    \end{cases}.
    \]
    It is easy to verify that $(\Z,\circ)$ is a group (isomorphic to the infinite dihedral group $D_\infty$)
    and that $(\Z,+,\circ)$ is a skew brace. We will denote this skew brace with $\CD_\infty$.

Note that the $\lambda$-map and $\theta$-map of $\CD_\infty$ are given by
\[
\theta_{(k,m)}(n)=\lambda_m(n)=\begin{cases}
        n, &\text{ if }m \text{ is even }\\
        -n &\text{ if }m \text{ is odd }
    \end{cases}
\]
\end{exa}

\section{The index of a sub skew brace}\label{secindex}

In this section, we study the index of sub skew braces. First, we show that if for a sub skew brace $A$ of $B$ both the additive and multiplicative index of $A$ in $B$ are finite, then $A$ contains a strong left ideal of finite index. This allows us to prove that if both the additive and multiplicative indices of a sub skew brace are finite, they need to coincide. This reduces the question to whether the finiteness of these indices are equivalent. We show that one implication holds for two-sided skew braces.

The relevance of left ideals in the context of determining the index of a sub skew brace is apparent in the following remark.

\begin{rem}\label{inv}
   Recall that in \cite{MR3647970} it was implicitly shown that for any left ideal $L$ and $a\in B$ the cosets determined by $a$ coincide, i.e. $ a \circ L = a + L$. Hence, for any left ideal $L$ of $B$, the index of $L$ in $B$ is well-defined. 
\end{rem}

It is a well-known and useful result in group theory that a subgroup of finite index has a normal subgroup of finite index contained in it. In the following proposition we extend this to skew braces.

\begin{pro}\label{sli2}
    Let $B$ be a skew brace and $A$ a sub skew brace such that both $|B: A|_+$ and $|B:A|_{\circ}$ are finite. Then, there exists a strong left ideal $L$ of $B$ of finite index contained in $A$.
\end{pro}
\begin{proof}
    First, let $T$ be a transversal of $A$ in $(B,\circ)$. Then we consider $$ L_1 = \bigcap_{t \in T} \lambda_t(A).$$ Note that $L_1$ is a finite intersection of subgroups of finite index in $(B,+)$, hence $L_1$ is of finite index in $(B,+)$. Then, we note for $b \in B$ that for every $t\in T$ there is a unique $t' \in T$ such that $b \circ t \in t' \circ A$ and vice versa. Hence, $$ \lambda_b\left( \bigcap_{t \in T} \lambda_t(A)\right) = \bigcap_{t \in T} \lambda_{b \circ t}(A) = \bigcap_{t' \in T} \lambda_{t'}(A). $$ Hence, we have shown that $L_1$ is a left ideal of finite index in the additive group. 
    However, for left ideals, additive and multiplicative cosets coincide, so $L_1$ is of finite index.

    Second, let $S$ be a transversal of $L_1$ in $(B,+)$. Consider the set $$ L = \bigcap_{s \in S} (s+L_1-s).$$ Then, $L$ is a subgroup of $(B,+)$ of finite index. Suppose that for $a \in B$ and $s,s',t \in S$ we have that $\lambda_a(s) = t+l$ for some $l \in L_1$ and $\lambda_a(s') = t+l'$ for some $l' \in L_1$. Then, $$ \lambda_a(-s+s') = -\lambda_a(s)+\lambda_{a}(s') = -l-t+t+l' = -l+l' \in L_1.$$ As $L_1$ is a left ideal, it is invariant under $\lambda_a^{-1}$. Thus, $s-s' \in L_1$, which shows that $s+L_1=s'+L_1$. Hence, for any $a \in B$ we have that $$ \lambda_a\left(\bigcap_{s \in S} (s+L_1-s)\right) = \bigcap_{s \in S} (\lambda_a(s)+L_1-\lambda_a(s)) = \bigcap_{s \in S} (s+L_1-s).$$ Hence, $L$ is a normal subgroup of finite index of $(B,+)$ and is a left ideal of $B$. As $L$ is a left ideal, \cref{inv} implies that the multiplicative index $|B:L|_{\circ}$  equals its additive index $|B:A|_+$, which shows the result.
\end{proof}

We use \cref{sli2} and the above remark to prove that if both the additive index and multiplicative index of a sub skew brace are finite, they have to coincide. 
\begin{thm}\label{finiteisnice}
    Let $(B,+,\circ)$ be a skew brace and let $(A,+,\circ)$ be a sub skew brace of $B$. Then, if $|B:A|_+$ and $|B:A|_\circ$ are finite, it holds that $$|B:A|_+=|B:A|_\circ.$$
\end{thm}
\begin{proof}
    By \cref{sli2}, there exists a strong left ideal $L$ of $B$ of finite index contained in $A$. Hence, it is also a strong left ideal of $A$. Hence, by \cref{inv} it follows that $$ |B:L|_+= |B:L|_{\circ} \text{ and } |A:L|_+ = |A:L|_{\circ}.$$ Hence, combining these equalities with the tower formula, one obtains that $$ |B:A|_\circ|A:L|_{\circ}=|B:L|_{\circ}=|B:L|_+=|B:A|_+|A:L|_+=|B:A|_+|A:L|_{\circ}.$$ As all indices involved are finite, it follows that  $|B:A|_+=|B:A|_\circ$.
\end{proof}

\begin{rem}
    \cref{finiteisnice} allows us to reformulate the definition of index for a sub skew brace given in \cite{index} in the following way. Given a skew brace $B$ and a sub skew brace $C$ of $B$ we say that $C$ has \textit{finite index} $B$ if there exists a positive integer $n$ such that $n=|B:C|_+=|B:C|_\circ$; we define the \textit{index} $|B:C|$ of $C$ in $B$ as $n$. If $C$ does not have finite index, we say that $C$ has \textit{infinite index}.
\end{rem}
For two-sided skew braces \cref{sli2} can be improved to provide an ideal of finite index.
\begin{cor}
    Let $B$ be a two-sided skew brace and $A$ a sub skew brace of finite index. Then there exists an ideal $I$ of $B$ of finite index contained in $A$.
\end{cor}
\begin{proof}
    First, by \cref{sli2}, there is a strong left ideal $L$ of $B$ of finite index contained in $A$. Consider $T$ a transversal of $(L,\circ)$ in $(B,\circ)$. Denote $$ I= \bigcap_{t\in T} \overline{t} \circ  L \circ  t.$$ Then, we claim that $I$ is the required ideal of $B$. Indeed, \cref{mult conj in two-sided} shows that multiplicative conjugation in a two-sided skew brace is a skew brace automorphism. Hence, $\overline{t}\circ L\circ t$ is a strong left ideal for all $t\in T$ of finite index. Hence, $I$ is a strong left ideal of $B$ of finite index. Moreover, by construction $(I,\circ)$ is a normal subgroup of $(B,\circ)$, which shows the result.
\end{proof}

Note that the previous proof strategy relies on the fact that multiplicative conjugation is an additive automorphism, which does not hold in general skew braces. Hence, the following question remains open.

\begin{question}
    Let $B$ be a skew brace and $A$ a sub skew brace of finite index. Does $A$ contain an ideal of $B$ of finite index?
\end{question}

Secondly, \cref{finiteisnice} reduces the question whether the index of a sub skew brace is well-defined to the following.

\begin{question}\label{equivalencefiniteindex}
    Let $B$ be a skew brace and $A$ a sub skew brace. Does it hold that $|B:A|_+<\infty$ is equivalent with $|B:A|_\circ< \infty$?
\end{question}

In the following proposition we show that two-sided skew braces of abelian type,
i.e. radical rings, the finiteness of the additive index implies the finiteness of the multiplicative index.

 \begin{pro}\label{atts}
    Let $(B,+,\circ)$ be a two-sided skew brace of abelian type and $A$ be a sub skew brace of $B$ such that $|B:A|_+<\infty$. Then $|B:A|_\circ=|B:A|_+$.
\end{pro}
\begin{proof}
    Consider the surjective action $$\varphi:A\to \End\frac{(B,+)}{(A,+)}$$ defined as $\varphi(a)=l_a$, where $l_a$ is the left multiplication by $a$. As $(B,+)/(A,+)$ is a finite abelian group, it follows that $\End(B,+)/(A,+)$ is finite. Observe that the left ideal of $B$ given by $$L=\{a\in A \ | a*b\in A, \text{ for any } \ b\in B\}$$ is $\ker \varphi$; thus $(A/L,+)$ is finite. Hence, from the tower formula, $L$ is a left ideal of $B$ of finite index, which shows that $|B:A|_\circ<\infty$ and the result follows from \cref{finiteisnice}.   
\end{proof}

By \cite{Trappeniers} a two-sided skew brace is an extension of a two-sided skew brace of abelian type
by a weakly trivial skew brace. To settle \cref{equivalencefiniteindex} for two-sided skew braces, we show that it holds for weakly trivial skew braces.

\begin{pro}\label{indexweaklytriv}
    Let $B$ be a weakly trivial skew brace and $A$ a sub skew brace. Then, if either the additive or multiplicative index of $A$ in $B$ is finite, then both are finite and equal.
\end{pro}
\begin{proof}
    Recall that a weakly trivial skew brace is a subdirect product of a trivial skew brace and an almost trivial skew brace. In particular, $B$ is a sub skew brace of $(B/B^2,\circ,\circ) \times (B/B^2_{op},\circ^{op},\circ)$. Denote $\pi_2$ the canonical epimorphism onto $B/B^2_{op}$. Let $A$ be a sub skew brace of $B$. Suppose that $|B:A|_+<\infty$. Then, $\pi(A)$ is of finite index in $B/B^2$. Hence, it contains an ideal $L$ of finite index. Consider $T=A \cap (B/B^2 \times L)$. Clearly, $|B:T|_+<\infty$. Note that for $(a,b) \in B$ and $(s,t)\in T$ one  has that $$ (a,b)+(s,t) = (a\circ s,b\circ^{op}t) = (a\circ s, t\circ b) = (a,b)\circ(s,\overline{b}\circ t \circ b). $$ Hence, $(a,b)+T=(a,b)\circ T$, which shows the result. 
\end{proof}

Through a theorem obtained by Trappeniers in \cite{Trappeniers}, weakly trivial skew braces and two-sided skew braces are intimately connected, which allows to prove the following.

\begin{cor}
    Let $(B,+,\circ)$ be a two-sided skew brace and let $A$ be a sub skew brace of $B$ such that $|B:A|_+<\infty$. Then $|B:A|_\circ =|B:A|_+$.
\end{cor}

\begin{proof}
    By \cite{Trappeniers} the ideal $I=B^2\cap B^2_{op}$ is a two-sided skew brace of abelian type and $B/I$ is weakly trivial. Hence, by \cref{indexdedekind},  \cref{atts} and \cref{indexweaklytriv} the result follows. 
\end{proof}

We shift focus from proving that the index is well-defined in certain classes of skew braces to proving that some particular sub skew braces always have well-defined index. 
 \begin{pro}\label{inker}
      Let $B$ be a skew brace and $x,y\in B$. Let $A \subseteq \ker \lambda $ be a sub skew brace of $B$. Then, $ A + x =A \circ x$. In particular, $ |B:A|_+ =|B:A|_{\circ} $.
 \end{pro}
 \begin{proof}
For any $a \in A$ one has that $\lambda_a=\id_B$. Hence, $$ a \circ x = a + \lambda_a(x) = a + x.$$
Thus, it follows that $A \circ x = A + x$. The result now follows.
 \end{proof}
 
We conclude this section with a first indication that the finiteness of the orbits under the $\lambda$-action are crucial to the study of skew braces of infinite order. It is easy to prove that the \cref{sli2} still holds if $(A,+,\circ)$ is an \textit{almost} $\lambda$-\textit{invariant} sub skew brace of B and $|B:A|_+$ is finite, where $A$ is called \textit{almost} $\lambda$-\textit{invariant} if $|B:Stab_\lambda(A)|_\circ$ is finite, or, equivalently, if the set $$\{\lambda_b(A) \ | \ b\in B\}$$ is finite. 

\begin{cor}
Let $(A,+,\circ)$ be an \textit{almost} $\lambda$-\textit{invariant} sub skew brace of a skew brace $(B,+,\circ)$. Then, if $|B:A|_+$ is finite, it holds that $$|B:A|_+=|B:A|_\circ.$$
\end{cor}

\begin{rem}
Note that, if $(A,+,\circ)$ is a sub skew brace of $(B,+,\circ)$ such that $|B:A|_\circ<\infty$, since $A\subseteq \Stab_\lambda(A)$, then $A$ is almost $\lambda$-invariant, while the converse is in general not true. However from the previous results, it follows that if $(A,+,\circ)$ is a sub skew brace of $(B,+,\circ)$ such that $|B:A|_+<\infty$ then it is equivalent to say that $|B:A|_\circ<\infty$ and that A is almost $\lambda$-invariant.
\end{rem}

\section{Skew braces with finite \texorpdfstring{$\lambda$}{lambda}-classes}\label{seclambda}

In this section, we examine the skew braces where the action of the group $\left<\lambda_x\mid x \in B \right>$ has finite orbits. As $\lambda$ intuitively corresponds to a conjugation, this is akin to the concept of finite conjugacy in groups.  First, we show some structural results, indicating that for this class $\ker\lambda$ plays an analogous role to the group theoretic center. Second, we show that these skew braces are finitely generated if and only if their additive and multiplicative groups are finitely generated. Lastly, we link this notion to almost $\lambda$-invariant sub skew braces.

\begin{defn}
    Let $B$ be a skew brace. An element $x\in B$ is a \emph{$\lambda_f$-element} if it has a finite $\lambda$-orbit, or equivalently, if $|B:\Stab_\lambda(x)|_\circ$ is finite. Moreover we set
\[
\lambda_f(B)=\{b\in B\mid b \text{ is a }\lambda_f\text{-element}\}.
\]
If $\lambda_f(B)=B$, then we say that $B$ is a $\lambda_f$-skew brace.
\end{defn}

First, we show that the set of $\lambda_f$-elements forms a left ideal. 

\begin{pro}\label{sli}
   Let $B$ be a skew brace. Then, $\lambda_f(B)$ is a left ideal. Furthermore, if $C$ is a skew brace and $\varphi:B\rightarrow  C$ an epimorphism, then $\varphi(\lambda_f(B))\subseteq \lambda_f(C)$.
\end{pro}
\begin{proof}
     
We first prove that $\lambda_f(B)$ is a subgroup of $(B,+)$. Observe that, since $[1]_{\lambda}=\{1\}$, $1$ is a $\lambda_f\text{-element}$. Moreover, if we consider $a,b\in\lambda_f(B)$ and $c\in B$, then $$\lambda_c(a-b)=\lambda_c(a)-\lambda_c(b)$$ and since we have finitely many possibilities for $\lambda_c(a)$ and $\lambda_c(b)$, it follows that $a-b\in\lambda_f(B)$. In order to prove that $\lambda_f(B)$ is $\lambda$-invariant, let us consider $d\in B$ and $x\in\lambda_f(B)$. Then we have that the set $$[\lambda_d(x)]_{\lambda}=\{\lambda_c(\lambda_d(x)) \ | \ c\in B\}=\{\lambda_{c\circ d}(x) \ | \ c\in B\}$$ is finite. 

Let $\varphi:B\rightarrow C$ be an epimorphism. Let $x \in \lambda_f(B)$. Then, for any $ c \in C$ there exists a $b \in B$ such that $\varphi(b)=c$. Then,  $$ \lambda_c(\varphi(x))=\lambda_{\varphi(b)}(\varphi(x))=\varphi(\lambda_b(x)).$$ Hence, as $[x]_{\lambda}$ is a finite set, it follows that $[\varphi(x)]_{\lambda}$ is a finite set, which shows the result.
\end{proof}

A non-trivial example of a $\lambda_f$-skew brace is the following.
\begin{exa}\label{exafree}
    Let $(F,+)$ be a free group on $\left\lbrace a,b \right\rbrace$.
    Denote $\lambda:F\rightarrow F$ the group automorphism induced by
    $\lambda(a)=b$ and $\lambda(b)=a$ 
    and $\varepsilon: F \rightarrow \mathbb{Z}$ the group morphism induced by 
    $\varepsilon(a)=1=\varepsilon(b)$.
    Then we define the operation $\circ$ in the following way:
    \[
    u\circ v=u+\lambda^{\varepsilon(u)}(v),
    \]
    for every $u,v\in F$.
    It is easy to check that with this definition, $(F,+,\circ)$ is a skew brace.
    
    Moreover, since $\lambda$ has order 2, for every $w\in F$,
    \[
    \lambda_w=\begin{cases}
        \id&\text{ if }\varepsilon(w)\text{ is even}\\
        \lambda&\text{ if }\varepsilon(w)\text{ is odd}
    \end{cases}.
    \]
    So the $\lambda$-orbit of an arbitrary element has at most $2$ elements.
\end{exa}

In the context of $FC$-groups, among the most important examples we find groups in which the center has finite index; the converse holds for finitely generated $FC$-groups \cite{neumann}. In the framework of skew  braces, we note that $\ker\lambda$ plays the role of the center in these results.

\begin{pro}\label{lambdafkerfinite}
    Let $(B,+,\circ)$ be a skew brace such that $|B:\ker\lambda|$ is finite. Then, each element of $B$ is a $\lambda_f$-element.
\end{pro}
\begin{proof}
    Let $t$ be $|B:\ker\lambda|$. Then, there exist $x_1,...,x_t$ elements of $B$ such that $$B=\bigcup_{i=1}^t\ker\lambda\circ x_i.$$ Let $a,b\in B$. Then, there exist $h\in\ker\lambda$ and $i\in \{1,...,t\}$ such that $a=h\circ x_i$. Thus $$\lambda_a(b)=\lambda_{h\circ x_i}(b)=\lambda_{x_i}(b)$$ and, since we have finitely many possibilities for $x_i$, $\left|[b]_\lambda\right|$ is finite.
\end{proof}

In the following result, we prove that the converse holds for skew braces of finitely-generated type.

\begin{lem}\label{lamf}
    Let $(B,+,\circ)$ be a skew brace such that the additive group $(B,+)$ is generated by a finite set of $\lambda_f$-elements $\{x_1,...,x_t\}$. Then, $|B:\ker\lambda|$ is finite. Moreover, if for all $i$, one has that $|[x_i]_{\lambda}| \leq k$, then $|B:\ker\lambda| \leq k^t. $
\end{lem}
\begin{proof}
    By hypothesis, $(B,+)=<x_1,...,x_t>_+$, where $|(B,\circ):\Stab_\lambda(x_i)|<\infty$ for each $i\in \{1,...,t\}$. Thus, if we show that $\ker\lambda=\Stab_\lambda(x_1)\cap...\cap\Stab_\lambda(x_t)$, the statement is proved. Clearly $\ker\lambda\leq\Stab_\lambda(x_1)\cap...\cap\Stab_\lambda(x_t)$; in order to prove the converse, let us consider $y\in \Stab_\lambda(x_1)\cap...\cap\Stab_\lambda(x_t)$ and $a\in B$. Then there exist $n_1,...,n_t\in \mathbb{Z}$ such that $a=n_1x_1+...+n_tx_t$. Therefore $\lambda_y(a)=\lambda_y(n_1x_1+...+n_tx_t)=n_1\lambda_y(x_1)+...+n_t\lambda_y(x_t)=n_1x_1+...+n_tx_t=a$ and $y\in\ker\lambda$. 
    Hence, it follows that $|B:\ker \lambda|<\infty$. Assume now that $|[x_i]_{\lambda}| \leq k$ for each $i\in\{1,...,t\}$, then 
    \[
    |B: \ker \lambda| =|B: \bigcap_{i=1}^t\Stab_{\lambda}(x_i)| \leq \prod_{i=1}^t|B:\Stab_{\lambda}(x_i)|=\prod_{i=1}^t|[x_i]_{\lambda}|\leq k^t.\qedhere
    \]
\end{proof}

    Note that $\cref{lamf}$ can not be extended to arbitrary $\lambda_f$-skew braces, as the following example shows.
\begin{exa}
    Denote $B$ the skew brace obtained in \cref{exafree}. Then, one obtains that $\ker \lambda= \left\lbrace w \mid \varepsilon(w)\in 2\mathbb{Z}\right\rbrace$. In particular, $a-b, a^2 \in  \ker \lambda$, which shows that $|B:\ker\lambda|_{\circ}=2$, in accordance with \cref{lamf}. Moreover, it can be easily seen that $\ker\lambda$ is an ideal of $B$.
    
    Then, $\bigoplus_{i\in \mathbb{N}} B_i$, where each $B_i \cong B$, is a $\lambda_f$-skew brace and $\ker \lambda = \bigoplus_{i\in \mathbb{N}} \ker \lambda_i$, where $\lambda_i$ denotes the $\lambda$-map of $B_i$. This shows that $$B/\ker\lambda \cong \bigoplus_{i\in \mathbb{N}} \mathbb{Z}_2,$$ where the right hand side is an infinite trivial skew brace.
\end{exa}

Note that in \cref{lamf} we explicitly mentioned and used that the additive group of the skew brace was finitely generated. For general skew braces, this is a stronger demand than that the skew brace itself is finitely generated. Recall that one says that a skew brace $B$ is finitely generated if there exists a finite subset $T$ such that the smallest skew brace containing $T$ is $B$. In the following series of results, we will show that these concepts coincide for the class of $\lambda_f$-skew braces. First, we give an explicit method to relate a set of generators of $B$ as skew brace to a set of generators of $(B,+)$ and $(B,\circ)$.

\begin{lem}
\label{gens}
    Let $B$ be a skew brace and let $T$ be a set of generators (as a skew brace) of $B$.
    Then, denoting $T'=\bigcup_{b\in B}\lambda_b(T)$,
    \[
    U=T' \cup (-T')
    \] 
    generates $(B,+)$ and $(B,\circ)$.
\end{lem}
\begin{proof}
Since the skew brace $B$ is generated by $T$, it is also generated by $U$.
Recall that the skew brace generated by the set $U$ is the union of the sets $U_n$,
where $U_1=U$ and
\[
 U_{n+1} 
 = \left\lbrace x\circ y, x+y, -x, \overline{x} \mid x,y \in U_n \right\rbrace \cup U_n.
 \]
 Hence, we proceed by induction.
 Clearly, every element of $U_1$ is the sum of elements of $U$. 
 Hence, assume that every element in $U_{n}$ is writable as a sum of elements of $U$.
 Let $x,y \in U_n$. 
 Then, clearly $x+y$ and $-x$ are sums of elements of $U$ 
 as $U$ is closed under taking additive inverses.
 Secondly, we note that $\overline{x}=-\lambda_{\overline{x}}(x)$. 
 As $x$ is a sum of elements of $U$ and the latter
 is closed under images of $\lambda_b$ for every $b\in B$ and additive inverses,
 one finds that $\overline{x}$ is a sum of elements of $U$.
 Similarly, we note that $x \circ y = x +\lambda_x(y)$,
 which is a sum of elements of $U$ for the same reasons.
 Hence, $(B,+)$ is generated as a group by $U$.
    
Finally, we note that if $$ x= u_1+ ... +u_k$$ for some $u_1,...,u_k \in U$
one finds that 
\[
x = u_1 \circ \lambda_{u_1}^{-1}(u_2)\circ ... \circ \lambda_{u_1\circ ... \circ u_{k-1}}^{-1}(u_k).
\]
As $U$ is closed under $\lambda$, it follows that 
this decomposition of $x$ is a product of elements of $U$. 
Hence, $U$ generates $(B,\circ)$, which shows the result.
\end{proof}

Hence, the fact that being finitely generated is non-ambiguous for $\lambda_f$-skew brace follows.

\begin{thm}\label{fg}
    Let $B$ be a $\lambda_f$-skew brace. Then, the following are equivalent.
    \begin{enumerate}
        \item The skew brace $B$ is finitely generated as a skew brace,
        \item the group $(B,\circ)$ is finitely generated,
        \item the group $(B,+)$ is finitely generated.
    \end{enumerate}
\end{thm}
\begin{proof}
    Clearly $(2)$ and $(3)$ imply $(1)$.
    Now we prove that $(1)$ implies $(2)$ and $(3)$.
    Let $T$ be a finite set of generators (as a skew brace) of $B$. 
    Then, denote $T'=\bigcup_{b\in B}\lambda_b(T)$ and $U=T' \cup (-T')$.
    As $T$ was finite and $B$ is $\lambda_f$, it follows that $T'$ is finite.
    Hence $U$ is finite and by \cref{gens} it generates $(B,+)$ and $(B,\circ)$.
\end{proof}
\cref{fg} allows us to rewrite \cref{lamf} to allow being finitely generated as a skew brace.
\begin{cor}
    Let $B$ be a skew brace. If $B$ is finitely generated by a set of $\lambda_f$-elements, then $|B:\ker\lambda|$ is finite.
\end{cor}
\begin{proof}
    If $B$ is generated by a set $T$ of $\lambda_f$-elements, then $B$ is contained in $\lambda_f(B)$, as the latter is a skew brace containing $T$. Thus, $B$ is a $\lambda_f$-skew brace. Hence, by \cref{fg}, $(B,+)$ is also finitely generated and the result follows from \cref{lamf}. 
\end{proof}

We record the following fact, as it is a straightforward consequence of combining \cref{lamf}, \cref{fg}, and the well-known fact that subgroups of finite index of a finitely generated group are finitely generated (cfr. \cite[Result 1.6.11]{robinson2}). 

\begin{pro}
    Let $B$ be a skew brace such that $B/\Soc(B)$ is finite. If $B$ is finitely generated, then $\Soc(B)$ is finitely generated.
\end{pro}

By \cref{inker} the index of $\Fix(B) \cap \ker \lambda$ in $B$ is well-defined. Hence, one obtains the following result, which can be viewed as a variant of Schur's theorem. Note that here we only obtain that $B^2$, the analog of the commutator subgroup, is only finitely generated, not necessarily finite. This discrepancy follows from the absence of demands on the additive center, compare with \cite{MR4256133}.

\begin{thm}
    Let $B$ be a skew brace.
    If $|B:\ker\lambda\cap\Fix(B)|$ is finite, then $(B^2,+)$ is finitely generated.
\end{thm}
\begin{proof}
    Let $t$ be the index of $\ker\lambda$ in $(B,\circ)$ and $s$ be the index of $\Fix(B)$ in $(B,+)$.
    Then there are $x_1,\dots,x_t\in B$ and $y_1,\dots, y_s\in B$ such that
    \[
    B=\bigcup_{i=1}^t \ker\lambda\circ x_i=\bigcup_{j=1}^s \Fix(B)+ y_j.
    \]
    Then for every $a,b\in B$ there exist $k\in \ker\lambda$, $f\in \Fix(B)$, $i\in\{1,\dots,t\}$ and $j\in\{1,\dots,s\}$ such that
    $a=x_i\circ k$ and $b=y_j+f$.
    In this way, 
    \[
    a * b = 
    \lambda_{x_i\circ k}(y_j+f)-(y_j+f)=
    \lambda_{x_i}(y_j)+f-f-y_j=
    \lambda_{x_i}-y_j=
    x_i*y_j.
    \]
    Therefore $(B^2,+)$ is finitely generated by 
    $\big\{x_i*y_j\mid i\in\{1,\dots,t\}, j\in\{1,\dots,s\} \big\}$.
\end{proof}
 
Clearly, the reverse does not hold. Indeed, consider the almost trivial skew brace on the infinite dihedral group $D_{\infty}$. Then, $\ker\lambda=\Fix(B)=Z(D_{\infty})=1$, but $[D_{\infty},D_{\infty}]$ is cyclic. Observe however that, if $B^2$ is finite, then $\lambda_f(B)=B$ since, for any fixed $d\in B$ $$\{\lambda_c(d)\ | \ c\in B\}=\{c*d+d\ | \ c\in B\}.$$ The reverse does not hold (see \cref{es1}).

The following results are akin to Baer's theorem, which shows that the quotient of an $FC$-group over its center is a residually finite torsion group \cite[Result 14.5.6]{robinson2}.
\begin{pro}
    Let $(B,+,\circ)$ be a $\lambda_f$-skew  brace. 
    Then 
    $$\frac{(B,\circ)}{(\ker\lambda,\circ)}$$
    is residually finite.
\end{pro}
\begin{proof}
    Since $B$ is a $\lambda_f$-skew brace, the index $|(B,\circ):\Stab_\lambda(x)|$ is finite for each $x\in B$ and $\ker\lambda=\bigcap_{x\in B}\Stab_\lambda(x)$. 
\end{proof}

As mentioned in \cref{secindex} the notion of almost $\lambda$-invariance of a sub skew brace $A$ in $B$ is closely related to the $\lambda_f$-elements of $B$. 
\begin{pro}
    Let $(B,+,\circ)$ be a skew brace and let $(A,+,\circ)$ be a finitely generated sub skew brace of $B$. Then
    \begin{enumerate}
        \item If $A$ is contained in $\lambda_f(B)$, then $A$ is almost $\lambda$-invariant in $B$.
        \item If $(A,+)$ is finite or cyclic and almost $\lambda$-invariant in $B$, then $A$ is contained in $\lambda_f(B)$.
    \end{enumerate}
\end{pro}
\begin{proof}
    \textit{1.} Let $A$ be a sub skew brace of $B$ contained in $\lambda_f(B)$. By \cref{fg}, the additive group of $A$ is finitely generated, say by $\left\lbrace x_1,..., x_t \right\rbrace$. Then $$\Stab_\lambda(x_1)\cap...\cap \Stab_\lambda(x_t)$$ is contained in $\Stab_\lambda(A)$ and has finite index in $(B,\circ)$. Thus, $|B:\Stab_\lambda(A)|_{\circ}$ is finite and $A$ is almost $\lambda$-invariant in $B$.

    \medskip
    
    \textit{2.} Assume now that $(A,+)$ is finite or cyclic and that $A$ is almost $\lambda$-invariant in $B$. By \cref{quotient}, the quotient group $(\Stab_\lambda(A)/\PStab_B(A),\circ)$ is isomorphic to a subgroup of $\Aut(A,+)$ which is finite. Thus $|B:\PStab_B(A)|_{\circ}$ is finite and the statement is proved.
    
\end{proof}

As a consequence, if all the sub skew braces of a skew  brace $B$ are almost $\lambda$-invariant, then $B=\lambda_f(B)$.

    Note that an almost $\lambda$-invariant sub skew brace of a skew brace $B$ with a finitely generated additive group is not necessarily contained in $\lambda_f(B)$.
\begin{exa}
    Consider, the almost trivial skew brace $B=\op\Triv(D_\infty)$
    on the infinite dihedral group.
    As noted in \cref{exaoptriv} the $\lambda$-action of $B$ coincides with the conjugation in $D_\infty$. Thus, $B$ is a sub skew brace of $B$ that is (almost) $\lambda$-invariant and has a finitely generated additive group, but $\lambda_f(B)=FC(D_\infty)$ does not contain $B$.
\end{exa}

\section{Skew braces with finite \texorpdfstring{$\theta$}{lambda}-classes}\label{sectheta}

Often in skew braces one combines a condition on skew braces with a similar condition on the additive group to result in properties that can be transported onto set-theoretic solution. In this section we introduce $\theta_f$-elements and $\theta_f$-skew braces, which correspond to $\lambda_f$-elements that are $FC$ in the additive group. We show that $\theta_f$-elements form a strong left ideal. In the remainder of the section we examine the structure of $\theta_f$-skew braces, showing that these behave similarly to $FC$-groups, where $\Soc(B)$ plays the role of the center. However, we will note some striking differences.

\begin{defn}
    Let $(B,+,\circ)$ be a skew brace and let $K$ be a subset of $B$. $K$ is called \emph{$\theta$-invariant} if $\theta_{(a,b)}(K)\subseteq K$, for every $a,b\in B$.
\end{defn}

\begin{pro}
    Let $(B,+,\circ)$ be a skew brace and let $K$ be a subgroup of $(B,+)$. Then, $K$ is a strong left ideal if and only if $K$ is $\theta$-invariant.
\end{pro}
\begin{proof}
    Suppose first that $K$ is a strong left ideal, so that, by definition, $K$ is $\lambda$-invariant and contains each conjugate of its elements respect to the $+$ operation. It follows that, for every $a,b\in B$, the set $$\theta_{(a,b)}(K)=\{a+\lambda_b(x)-a\mid x \in K\}$$ is contained in $K$. The converse is a direct consequence of the fact that $\theta_{(1,b)}(K)=\lambda_b(K)$ and $\theta_{(a,1)}(K)=K^a$.
\end{proof}

For a skew brace $B$ the following lemma explains how to find the smallest strong left ideal containing a subgroup $(K,+)$ of $(B,+)$. Note that applying this to a trivial skew brace, one obtains the well-known link between a subgroup and its normal closure.

\begin{lem}\label{leftidealgenerated}
    Let $(B,+,\circ)$ be a skew brace and let $K$ be a subgroup of $(B,+)$. Then, the strong left ideal generated by $K$ is the subgroup $$\left<\theta_{(a,b)}(K) \mid a,b\in B\right>_+.$$
\end{lem}
\begin{proof}
    Clearly, the subgroup $A=\left<\theta_{(a,b)}(K) \mid a,b\in B\right>_+$ is contained in the strong left ideal generated by $K$. Now we need to prove that $A$ is $\lambda$-invariant in $B$ and that it is a normal subgroup of $(B,+)$. In order to do that, let us consider $x\in A$ and $a\in B$. Then $x=x_1+...+x_t$, with $x_i\in\theta_{(a_i,b_i)}(K)$ for some $a_i,b_i\in B$. Thus $\lambda_a(x)=\lambda_a(x_1)+...+\lambda_a(x_t)\in A$, since $\lambda_a(x_i)\in \theta_{(\lambda_a(a_i),a\circ b_i)}(K)$ and 
    \[
    c+x-c=c+x_1-c+c+x_2-c+...+c+x_t-c\in A,
    \]
    since $c+x_i-c\in \theta_{(c+a_i,b_i)}(K)$.
\end{proof}

As \cref{leftidealgenerated} indicated and remarked in \cite{conju}, the elements that are in the same $\theta$-orbit can intuitively be thought of as conjugates. Hence, the following definition can be viewed as an analogue of the $FC$ center in group theory.

\begin{defn}
    Given a skew brace $B$, an element $x\in B$ is a \emph{$\theta_f$-element} if it has a finite $\theta$-orbit and we set
\[
\theta_f(B)=\{b\in B\mid b \text{ is a }\theta_f\text{-element}\}.
\] 
Moreover, if $B=\theta_f(B)$, we call $B$ a \emph{$\theta_f$-skew brace}.
\end{defn}
 Note that each $\theta_f$-element is an element that is $\lambda_f$ and $FC$ in $(B,+)$. Similar to $FC$-groups, we distinguish the class of skew braces where the size of the $\theta$-classes is uniformly bounded.
 
\begin{defn}
    A skew brace $B$ is called 
    $\theta_{Bf}$ if there exists a 
    positive integer $n$ such that
    $|[b]_\theta|\leq n$ for each $b\in B$.
\end{defn}

Note that actually this is not an analog of $FC$-groups, but an extension. 

\begin{exa}
    
Let $(G,\cdot)$ be a group then, $FC(G,\cdot)=\theta_f(G,\cdot,\cdot)=\theta_f(G,\cdot^{op},\cdot)$ and one finds that $G$ is $FC$ if and only if the (almost) trivial skew brace on $G$ is a $\theta_f$-skew brace.
\end{exa}

A large family of non-trivial examples is given by the structure skew braces of finite solutions. This can either be shown with direct computations, or will follow more elegantly by \cite[Theorem 6.6]{MR3974961} from \cref{oversocfinite}.

\begin{exa}
Let $(X,r)$ be a finite solution. Then, the structure skew brace $B(X,r)$ is $\theta_{Bf}$.
\end{exa}

\begin{pro}
\label{theta_f strong left ideal}
   Let $B$ be a skew brace, then $\theta_f(B)=\lambda_f(B) \cap FC(B,+)$ is a strong left ideal. Moreover, if $C$ is a skew brace and $\varphi:B\rightarrow C$ a skew brace epimorphism, then $\varphi(\theta_f(B)) \subseteq \theta_f(C)$.
\end{pro}
\begin{proof}
    As both $\lambda_f(B)$ and $FC(B,+)$ are strong left ideals of $B$, the first claim follows once the equality is shown. First, we remark that clearly any $\theta_f$-element is both $\lambda_f$ and additively $FC$. Hence, it rests to show the reverse inclusion. Let $x \in \lambda_f(B) \cap FC(B,+)$. Denoting $x_1,...,x_k$ the elements of the $\lambda$-class of $x$, we find that $[x]_{\theta} = \cup_{i=1}^k [x_i]_{+}.$ Since $\lambda_f(B) \cap FC(B,+)$ is a strong left ideal, the number of elements in the additive conjugacy class of any $x_i$ is finite, which shows that $|[x]_{\theta}|< \infty$. Hence, the inclusion is shown, which shows the first claim.

    Let $c,a \in C$ and $b\in \theta_f(B)$. Then, there exist $c_1,a_1 \in B$ such that $\varphi(c_1)=c$ and $\varphi(a_1)=a$. As $\varphi$ is a skew brace morphism \begin{align*} \theta_{(c,a)}(\varphi(b))=c+\lambda_a(\varphi(b))-c&=\varphi(c_1)+\lambda_{\varphi(a_1)}(\varphi(b))-\varphi(c_1)\\&=\varphi(c_1+\lambda_{a_1}(b)-c_1)\\ &=\varphi(\theta_{(c_1,a_1)}(b)).\end{align*}
    Hence, $$ [\varphi(b)]_{\theta}=\varphi([b]_{\theta}),$$ which shows the inclusion $\varphi(\theta_f(B)) \subseteq \theta_f(C)$. 
\end{proof}

 In the following result, we prove that the $\theta_f$-elements form an ideal in a two-sided skew brace.

\begin{pro}
    Let $B$ be a two-sided skew brace. Then $\theta_f(B)$ is an ideal.
\end{pro}
\begin{proof}
    By \cref{theta_f strong left ideal} $\theta_f(B)$ is a strong left ideal.

    It remains to show that $\theta_f(B)$ is a normal subgroup in $(B,\circ)$.
    Let $x\in \theta_f(B)$, $c\in B$ and 
    $\varphi_c$ be the multiplicative conjugation by $c$.
    Since $B$ is two-sided, by \cref{mult conj in two-sided}, we know that $\varphi_c\in\Aut(B,+,\circ)$.
    Hence 
    \[
    \theta_{(a,b)}(c\circ x\circ \overline{c})=
    \theta_{(a,b)}\big(\varphi_c(x))=
    \varphi_c\left(\theta_{\left(\varphi_c^{-1}(a),\varphi_c^{-1}(b)\right)}(x)\right)
    \]
    So
    $[c\circ x\circ \overline{c}]_\theta=
    [\varphi_c(x)]_\theta
    =\varphi([x]_\theta)$ and it is finite.
\end{proof}

Observe that the set of the $\theta_f$-elements of a skew brace is generally not an ideal. 
\begin{exa}
Let $B$ be the skew brace given in \cref{rosita}. Then,
in this case we have 
\begin{align*}
    \Soc(B)&=\ker\lambda=\Z/3\times\Q\times\{0\}\\
\lambda_f(B)&=\Z/3\times \{0\}\times\Z\\
\Ann(B)&=Z(B,\circ)=\{(0,0,0)\}.
\end{align*}

Hence, $\lambda_f(B)=\theta_f(B)$ is not an ideal. Indeed, consider $(a,0,k)\in \lambda_f(B)$ and $(b,y,l)\in B$. Then

\begin{align*}
   \overline{(b,y,l)}\circ (a,0,k)\circ (b,y,l)  &=((-1)^{l+1}b,-2^{-l}y,-l)\circ (a+(-1)^kb,2^ky,k+l)\\
&=((-1)^{l+1}b+(-1)^{-l}(a+(-1)^kb),2^{-l}y(2^k-1),k) \\
&\notin \lambda_f(B) \text{ for } k,y\neq 0.
\end{align*}
\end{exa}

\begin{cor}\label{fgtheta}
    Let $B$ be a skew brace. Then, if $(B,+)$ is finitely generated by $\theta_f$-elements, then $B$ is a $\theta_f$-skew brace. Furthermore, if $B$ is a $\theta_f$-skew brace, then the following are equivalent \begin{itemize}
        \item the skew brace $B$ is finitely generated,
        \item the group $(B,+)$ is finitely generated,
        \item the group $(B,\circ)$ is finitely generated.
    \end{itemize}
\end{cor}
\begin{proof}
    If $(B,+)$ is generated by a finite set of $\theta_f$-elements, it follows that it is generated by a finite set of $FC$-elements. Hence, $(B,+)$ is $FC$. Moreover, $(B,+)$ is finitely generated by $\lambda_f$-elements, which entails that $\lambda_f(B)=B$. Hence, $B$ is $\theta_f$. The remainder now follows from \cref{fg}.
\end{proof}

Similarly to how a central-by-finite group is (bounded) $FC$, one finds that a socle-by-finite skew brace is $\theta_{bf}$.

\begin{lem}\label{oversocfinite}
    Let $(B,+,\circ)$ be a skew brace such that $|B:\Soc(B)|$ is finite. Then, $B$ is a $\theta_{Bf}$-skew brace.
\end{lem}

\begin{proof}
    By our hypothesis $|(B,\circ):\ker\lambda|$ and $|(B,+):Z(B,+)|$ are finite, thus there exist two positive integers $n$ and $m$ such that $|(B,\circ):\ker\lambda|=n$ and $|(B,+):Z(B,+)|=m$. Then there exist $x_1,...,x_n,y_1,...,y_m\in B$ such that $$B=\bigcup_{i=1}^n\ker\lambda\circ x_i=\bigcup_{j=1}^mZ(B,+)+y_j.$$ Let $a,b,c$ be elements of $B$. Then there exist $b_1\in\ker\lambda$, $a_1\in Z(B,+)$, $i\in\{1,...,n\}$ and $j\in\{1,...,m\}$ such that $a=a_1+y_j$ and $b=b_1\circ x_i$. It follows that 
    $$\theta_{(a,b)}(c)=\theta_{(a_1+y_j,b_1\circ x_i)}(c)=y_j+\lambda_{x_i}(c)-y_j.$$ 
    So 
    \[ |[c]_\theta|=|\{y_j+\lambda_{x_i}(c)-y_j\mid i\in\{1,...,n\}\text{ and }j\in\{1,...,m\}\}|\leq mn,
    \]
   which proves the result.
\end{proof}

In the following result, we will show that the converse holds for braces of finitely generated-type.

\begin{lem}
\label{thetafg}
    Let $(B,+,\circ)$ be a skew brace such that the additive group $(B,+)$ is generated by a finite set of $\theta_f$-elements $\{x_1,...,x_t\}$. Then, $|B:\Soc(B)|$ is finite. Moreover, if for all $i$ one has that $|[x_i]_{\theta}| \leq k$, then $|B:\Soc(B)| \leq k^{t}.$
\end{lem}

\begin{proof}
    By \cref{inker} and \cref{lamf} it results that $$|(B,+):\ker\lambda| =|(B,\circ):\ker\lambda|\leq \prod_{i=1}^t|[x_i]_{\lambda}|.$$ Moreover, since each $\theta_f$-element of $B$ is an $FC$-element of $(B,+)$ we have that $|(B,+):Z(B,+)|\leq \prod_{i=1}^t|[x_i]_{+}|$. Hence, the statement follows as \begin{align*} | B : \Soc(B)| &\leq |(B,+):\ker \lambda| |(B,+):Z(B,+)|\\  
    &\leq \prod_{i=1}^t|[x_i]_{\lambda}| \prod_{i=1}^t|[x_i]_{+}|\\ 
    &=\prod_{i=1}^t|[x_i]_{\theta}|\\
    &\leq k^t.\qedhere\end{align*} 
\end{proof}

Note that the notion of $\theta_f$-elements of a skew brace $B$, by definition, corresponds to the finiteness of its orbit under the $\theta$-action of $(B,+)\rtimes_{\lambda}(B,\circ)$ on $(B,+)$. However, it is natural to consider when elements of $(B,+)\rtimes_{\lambda}(B,\circ)$ are $FC$. This led to the introduction of $(s)$-elements by Colazzo, Ferrara and Trombetti \cite{s}. 

\begin{defn}
    Let $B$ be a skew  brace. Then an element $x \in B$ is an $(s)$-element, if $$ |\left\lbrace g*x ,x*g, \overline{g} \circ x \circ g, -g+x+g \mid g \in B\right\rbrace| < \infty.$$
\end{defn}
Note that every $(s)$-element is necessarily $\theta_f$ and an $FC$-element of $(B,\circ)$. 
\begin{thm}
    Let $(B,+,\circ)$ be a skew  brace and let $x\in \ker\lambda$. Then, if $x$ is a $\theta_f$-element of $B$, $x$ is an $FC$-element of $(B,\circ)$. In particular, $x$ is an $(s)$-element.  
\end{thm}
\begin{proof}
        Let us consider $a\in B$ and $x\in \ker\lambda$. Then 
        \[
        a\circ x \circ \overline{a}
        =a+\lambda_a(x)-a
        =\theta_{(a,a)}(x).
        \]
        Thus $[x]_\circ=\{\theta_{(a,a)}(x) \ | \ a\in B\}\subseteq [x]_\theta$
        and the statement is proved.
\end{proof}
One also finds that $(s)$-elements correspond to elements in the $FC$ center of $(B,\circ)$ for $\theta_f$-skew braces.

\begin{lem}
\label{bound x*g}
Let $B$ be a $\theta_{f}$-skew brace such that 
$(B,\circ)$ is an $FC$-group. Then, for any $x\in B$, the number of elements of the form $x*g$ is finite. 

Moreover, if for every $b\in B$ it holds that $|[b]_\circ |\leq n$, i.e. $B$ is $\theta_{bf}$ and for every $x\in B$ one has that $|[x]_\theta|\leq m$, i.e. $(B,\circ)$ is BFC, then 
  \[
  |\left\lbrace x*g \mid g \in B\right\rbrace|<mn.
  \]
\end{lem}
\begin{proof}
For every $g\in B$
     \begin{align*}
        x*g &= -x+x \circ g - g\\ &=-x + g \circ (\overline{g} \circ x \circ g) -g \\
        &= -x+g-g+g\circ (\overline{g} \circ x \circ g) - g \\ &= -x+g+\lambda_g(\overline{g} \circ x \circ g)-g.
    \end{align*}

    Hence 
    \begin{align*}
        \{x*g\mid g\in B\}&
        =\{x+g+\lambda_g(\overline{g} \circ x \circ g)-g\mid g\in B\}\\
        &=\{x+g+\lambda_g(y)-g\mid y\in[x]_\circ , g\in B\}\\
        &=\{x+\theta_{g,g}(y)\mid y\in[x]_\circ, g\in B\}\\
        &=\{x+z\mid z\in[y]_\theta, y\in[x]_\circ\}
    \end{align*}
    and 
    \[
    |\{x*g\mid g\in B\}|\leq
    \left|\{z\mid z\in[y]_\theta, y\in[x]_\circ\}\right|=
    \left|\bigcup_{y\in [x]_\circ}[y]_\theta\right|<\infty,
    \]
    where the last inequality follows as it is a finite union of finite sets. Moreover, if for every $b\in B$ one has that $|[b]_\circ |\leq n$ and for every $x\in B$ one has that $|[x]_\theta|\leq m$, then the last inequality can be sharpened to
  $$ \left|\bigcup_{y\in [x]_\circ}[y]_\theta\right|\leq mn, $$
  which shows the result.
\end{proof}

\begin{pro}
\label{S = thetaf and mult FC}
    Let $B$ be a $\theta_f$-skew brace. Then, an element $x \in B$ is contained in $FC(B,\circ)$ if and only if $x$ is an $(s)$-element. Moreover, $B$ is an $(S)$-skew brace if and only if $B$ is a $\theta_f$-skew brace with $(B,\circ)$ an $FC$ group.
\end{pro}
\begin{proof}
    First, let $x \in B$ be an $(s)$-element, then by definition $x \in FC(B,\circ)$.

    Second, let $x \in FC(B,\circ)$ and let $n\in \N$ such that
    $|[x]_\circ |\leq n$. 
    Moreover, since $B$ is $\theta_f$, there is $m\in \N$ such that
    $|[x]_\theta|\leq m$.
    Then, by \cref{bound x*g},
    \[
      |\left\lbrace x*g \mid g \in B\right\rbrace|<mn.
    \]
    As $x$ is a $\theta_f$-element and only has finitely many multiplicative conjugates, it follows that the set $\left\lbrace g*x,x*g,-g+x+g, \overline{g} \circ x \circ g\right\rbrace $ only has finitely many elements. Thus, $x$ is an $(s)$-element, which shows the result.
\end{proof}

Examples of $\theta_f$-skew braces that are not $(S)$ are legion.

\begin{exa}
Let $CD_{\infty}$ be the skew brace defined in \cref{es1}. Then, $\CD_\infty$ is a $\theta_f$-skew brace,
$\Soc(\CD_\infty)=\ker\lambda=2\Z$ has finite index, but $Z(\CD_\infty,\circ)=\{0\}$, 
so $\Ann(\CD_\infty)=\{0\}$ doesn't have finite index, so $\CD_\infty$ is not an $(S)$-brace.
\end{exa}

Moreover, structure groups of finite involutive solutions provide a family of examples. Indeed, for any finite set $X$ there is a unique involutive solution for which the structure skew brace has property $(S)$.

\begin{exa}
    Let $(X,r)$ be a finite involutive solution. Suppose that the natural skew brace structure on $G(X,r)$ is $(S)$. Then, the group $G(X,r)$ is torsion-free and central-by-finite, which shows it is free abelian. Hence, by \cite{MR3974961}, the solution $(X,r)$ is the flip solution.
\end{exa}

As one may expect from the equivalence of $(S)$ with $\theta_f$-skew braces with an $FC$-group as multiplicative group, this equivalence extends to the bounded versions.
\begin{pro}
    Let $B$ be a skew brace.
    Then $B$ is $(BS)$ if and only if $B$ is a $\theta_{Bf}$-skew brace and 
    $(B,\circ)$ is a BFC-group.
\end{pro}
\begin{proof}
Given a skew brace $B$ and $x\in B$, 
we denote 
\[
B_x=\{g*x,x*g,-g+x+g, \overline{g} \circ x \circ g\}.
\]

First, suppose that $B$ is $(BS)$.
Then, there exists $n\in\N$ such that
\[
|B_x|\leq n,
\]
for every $x\in B$. As 
$[x]_\theta=\bigcup_{b\in B}[\lambda_b(x)]_+$, one finds that $|[x]_\theta|\leq n^2$. Thus $B$ is $\theta_{bf}$. As $[x]_\circ\subseteq B_x$,
it follows directly that $|[x]_\circ|\leq n$ for every $x\in B$.
Hence, $(B,\circ)$ is a BFC-group.

Vice versa, suppose that $B$ is $\theta_{Bf}$ and $(B,\circ)$ is a BFC-group.
Let $m,n\in \N$ such that $|[x]_\circ |\leq n$ and $|[x]_\theta|\leq m$.
Then, by \cref{bound x*g}, 
\[
|\left\lbrace x*g \mid g \in B\right\rbrace|<mn.
\]
Since 
\[
B_x=\{g*x \mid g \in B \}\cup\{x*g \mid g \in B\}\cup [x]_+\cup [x]_\circ
\]
 and $ |\{g*x \mid g \in B \}|=|[x]_\lambda|\leq [x]_\theta\leq m$,
\[
|B_x|\leq m+mn+|[x]_+|+|[x]_\circ|\leq m+mn+m+n,
\]
for all $x\in B$.
Therefore $B$ is $(BS)$.
\end{proof}

The proof of \cref{S = thetaf and mult FC} is highly based on the fact that the multiplicative conjugates of a $\theta_f$-element are themselves $\theta_f$, which holds in a $\theta_f$-skew brace, but not for $\theta_f$-elements in an arbitrary skew brace. Thus, the following is a natural question.
\begin{question}
    Let $B$ be a skew brace and $x \in B$. If $x \in FC(B,+) \cap FC(B,\circ) \cap \lambda_f(B)$, is then $x$ an $(s)$-element? The converse clearly holds.
\end{question}
As the $\lambda$-action on an element of the socle of a skew brace corresponds to multiplicative conjugation, a link between $\lambda_f$-elements and $FC$-elements of $(B,\circ)$ is found. 
\begin{thm}
    Let $(B,+,\circ)$ be a skew  brace and let $x$ be an element of $\Soc(B)$. Then $x$ is an $FC$-element of $(B,\circ)$ if and only if $x$ is a $\lambda_f$-element. In particular, every element in $\Soc(B)$ is a $\lambda_f$-element if and only if $(\Soc(B),\circ)$ is contained in the $FC$-center of $(B,\circ)$. Moreover, any $\lambda_f$-element of $\Soc(B)$ is an $(s)$-element of $B$.
\end{thm}
\begin{proof}
    Let $x$ be an element of $\Soc(B)$ and let $a$ be an element of $B$.
    Then 
    \[
    a\circ x\circ\overline{a} 
    =a\circ (x+\lambda_x(\overline{a}))
    =a+\lambda_a(x)+\lambda_a(\overline{a})
    =a+\lambda_a(x)-a=\lambda_a(x).
    \]

    Thus $[x]_\circ=[x]_\lambda$ and the statement is proved.
\end{proof}

Any left ideal of a skew brace $B$ contained in $\Soc(B)$ is an ideal. Hence, we note the following.

\begin{cor}
    The set $\Soc(B)\cap \lambda_f(B)=\Soc(B)\cap \theta_f(B)$ is an ideal of $B$ consisting of $(s)$-elements.
\end{cor}

Note that generally the ideal $\Soc(B) \cap \theta_f(B)$ does not coincide with $\Soc(B)$ or $\Ann(B)$ as the following example shows.
\begin{exa}
    Let $B$ be the skew brace given in \cref{rosita}. Then,  $$\Ann(B)\subsetneq \theta_f(B)\cap  \Soc(B)\subsetneq  \Soc(B).$$
\end{exa}

The following extends \cite[Lemma 4]{FC} to skew braces. Note that the seemingly extra condition that $\Soc(B)\subseteq \theta_f(B)$ is automatically satisfied for (almost) trivial skew braces.
\begin{lem}\label{lemma4}
    Let $B$ be a skew brace and $L$ be a strong left ideal of $B$, that is finitely generated as skew brace. If $\Soc(B) \subseteq \theta_f(B)$ and $|L: L\cap \Soc(B)|<\infty$, then $L \subseteq\theta_f(B)$.
\end{lem}
\begin{proof}
    Denote $N=L \cap \Soc(B)$. Clearly, $|L/\Soc(L)|<\infty$, which shows that $L$ is $\theta_f$. Hence, by \cref{fgtheta}, it follows that $(L,+)$ is finitely generated. Hence, $(N,+)$ is a finitely generated abelian group. In particular, the equation $mz = y$ has only finitely many solutions $z \in N$ for given $y \in N$ and positive integer $m$.
    Suppose to have $t \in L$ and $b \in B$ such that $\lambda_b(t) =t +z$ for some $z \in L\cap \Soc(B)$. There exists a positive integer $m$ such that $mt \in N$. As $mt \in \Soc(B)\subseteq \theta_f(B)$, it follows that there is a (finite) transversal $T$ for $\Stab_{\lambda}(mt)$ in $B$. Let $b = a \circ c$, where $a \in T$ and $c \in \Stab_{\lambda}(mt)$. Then, using that $z \in Z(B,+)$ one finds that
    $$ \lambda_a(mt)=\lambda_{a \circ c}(mt) = \lambda_{b}(mt) =m\lambda_{b}(t)=mt+mz.$$ 
    This is equivalent to \begin{equation}\label{equat}
        \lambda_a(mt)-mt = mz.
    \end{equation}
    Hence, for given $a$ and $t$ there are only finitely many $z$ that satisfy this equation. As $a \in T$ and the latter is finite, it follows that the number of $z$ satisfying any of the equations of the form \eqref{equat}, where $a \in T$, is finite. Hence, there are only finitely many different $\lambda_b(t)$ such that $\lambda_b(t) =t+z$ with $z \in N$. 
    Consider the cosets $c_1+N,\dots, c_k+N$ in $L$. Then, for every $b \in B$ there exists an $1\leq i \le k$ such that $\lambda_b(t) \in c_i+N$. 
    Let $\lambda_g(t),\lambda_h(t) \in c_i +N$ for some $g,h \in B$. Then, there is a $z \in N$ such that $$ \lambda_g(t) =\lambda_h(t) +z.$$ Equivalently, with $z'=\lambda_{\overline{h}}(z)$ one has that $$ \lambda_{\overline{h}\circ g}(t) =t+z'.$$ 
    As proven before, the number of $z'$ is finite. As $z' \in \Soc(B) \subseteq \lambda_f(B)$, it follows that $|[z']_{\lambda}|<\infty$. Thus, for a given $h \in B$ there can be only finitely many different $\lambda_g(t)$ in the same coset as $\lambda_h(t)$.  As the number of cosets is finite, it follows that $[t]_{\lambda}$ is a finite set. Hence, $L \subseteq \lambda_f(B)$.
    As $\Soc(B) \subseteq Z(B,+)$, it follows from \cite[Lemma 4]{FC} that $L \subseteq FC(B,+)$, which shows that $L \subseteq \theta_f(B)$, proving the result.
\end{proof}

One of the most important results in the study of $FC$-groups is Dietzmann’s Lemma \cite[Result 14.5.7]{robinson2}, which guarantees that a finite set of periodic $FC$-elements is contained in a finite normal subgroup. The following result extends this to  $\theta_f$-skew braces. We show that any finite set of (additively) periodic elements is contained in a strong left ideal.

\begin{thm}\label{dietz}
    Let $B$ be a skew brace and $x_1,...,x_n \in B$ be $\theta_f$-elements of finite additive order. Then, they generate a strong left ideal of finite order. 
\end{thm}
\begin{proof}
    Consider the subset $X = \left\lbrace \theta_{a,b}(x_i) \mid a,b \in B, 1 \leq i \leq n\right\rbrace$. As the elements $x_i$ are $\theta_f$-elements, it follows that $X$ is a finite set of elements of finite additive order that is closed under the $\theta$-action. Hence, the subgroup $L$ generated by these elements is finite, as the generating set is closed under conjugation and is closed under the $\theta$-action. Hence, $L$ is a finite strong left ideal that contains the $x_i$. 
\end{proof}

Note that in \cref{dietz} we explicitly demanded that the elements have finite additive order. However, they are contained in a finite strong left ideal, which entails that they need to have finite multiplicative order. Hence, the following result is shown.
\begin{cor}
\label{Dietzmann cor}
    Let $B$ be a skew brace and let $x$ be a $\theta_f$-element of finite additive order. Then $x$ has finite multiplicative order.

    Furthermore, if $B$ is a $\theta_f$-skew brace, then the torsion $T_+(B)$ of $(B,+)$ forms a locally finite strong left ideal that is contained in the set of torsion elements of $(B,\circ)$.
\end{cor}

Unfortunately, generally in a $\theta_f$-skew brace $B$ the set of  torsion elements of $(B,\circ)$ is not a subgroup, as demonstrated by the following example

\begin{exa}
Recall the skew brace $CD_{\infty}$ defined in \cref{es1}.

The set of torsion elements of $(\Z,\circ)$ is $\{0\}\cup (2\Z+1)$, which is not a subgroup and whose group generated by is the full $\Z$. Moreover, we observe that $\Z^{(2)}$ is not periodic.

\end{exa}

The following results provide an analog of \textit{Dietzmann's Lemma} for two-sided skew braces.

\begin{thm}
    Let $x_1,...,x_n$ be $\theta_f$, additively periodic elements of a two-sided brace $(B,+,\circ)$. Then $x_1,...,x_n$ are contained in a locally finite ideal of $B$. If furthermore $x_1,...,x_n \in FC(B,\circ)$, then they are contained in a finite ideal of $B$.
\end{thm}
\begin{proof}
    Let $a_1,...,a_t$ be elements of the ideal $I$ generated by $x_1,...,x_n$. Then, by \cref{dietz}, there exists a finite strong left ideal $A$ of $B$ such that $x_1,...,x_n\subseteq A\subseteq I$. Observe that $I=\left<\overline{g}   \circ A\circ g \ | \ g\in B\right>_+$, thus each element of $I$ is additively periodic and the group $\left<a_1,...,a_t\right>_+$ is finite. The statement follows by \cref{fg}. Assume now that $x_1,...,x_n \in FC(B,\circ)$. Then $A\subseteq FC(B,\circ)$, in particular $\{\overline{g}\circ A\circ g \ |\ g\in B\}$ is finite, thus $I$ is finite, since it is a finitely generated subgroup of $(B,+)$ consisting of periodic $FC$-elements of $(B,+)$.
\end{proof}

In the following result we focus on a particular subclass of $\theta_f$-skew braces $B$ where the additive group $(B,+)$ is $FC$ and for every $b$ the automorphism $\lambda_b$ is of finite order. 

First, we provide some technical lemmas which will allow for a sharper result for $\theta_{bf}$ skew braces. The following lemma is well-known in group theory. We include the proof to indicate the similarity with the proof of  \cref{order of lambda for theta Bf}.

\begin{lem}
\label{BFC -> G/Z has finite exp}
    Let $(G,\cdot)$ be a BFC-group with $k\in\N$ such that
    $|[g]_\cdot|\leq k$ for all $g\in G$.
    Then $G/Z(G)$ is of finite exponent at most $k!$.
\end{lem}
\begin{proof}
    Consider the monomorphism $$\phi:xZ(G)\in G/Z(G)\to (x(C_G(g))_G)_{g\in G}\in Cr_{g\in G}(C_G(g))_G.$$ The statement follows from the fact that the group $Cr_{g\in G}(C_G(g))_G$ has finite exponent of order $\leq k!$.
\end{proof}

For a $\theta_{bf}$ skew brace, we provide a bound on the order of the automorphism $\lambda_b$ for any $b \in B$.

\begin{lem}
\label{order of lambda for theta Bf}
    Let $(B,+,\circ)$ be a $\theta_{Bf}$ skew brace
    with $k\in\N$ such that
    $|[x]_\theta|\leq k$ for all $x\in B$.
    Then $\lambda_b$ has finite order $\leq k!$.
\end{lem}
\begin{proof}
  Let $b\in B$. Since $|[x]_\lambda|\leq k$, 
  for each $x\in B$,
  there exists $k_x\in\{1,2,\dots, k\}$ such that
  $\lambda_{b^{k_x}}(x)=\lambda^{k_x}_b(x)=x$.
  In particular, $\lambda^{k!}_b(x)=(x)$ for every $x\in B$.
\end{proof}

The subclass of $\theta_f$-skew braces where for all $b \in B$ the automorphism $\lambda_b$ has finite order share a striking resemblance to $FC$-groups. Indeed, the following results are akin to Baer's theorem, which shows that the quotient of an $FC$-group over its center is locally finite.

\begin{thm}\label{unsatisfacttheorem}
    Let $B$ be a $\theta_f$-skew brace.
    Then every $\lambda_b$ is of finite order if and only if $B/\Soc(B)$ is periodic. 
    Moreover, if $B$ is $\theta_{bf}$, then $B/\Soc(B)$ is of finite exponent.
\end{thm}
\begin{proof}
Let $B$ be a $\theta_f$-skew brace and  $x \in B$.
Then, we know that $\Stab_{\lambda}(x) $ is 
of finite index $\leq [x]_\theta=n$
in $(B,\circ)$.
Take $T=\left\lbrace t_1,..., t_n \right\rbrace$ to be a transversal. 
Then, the skew brace $A=\left<t_1,...,t_n,x\right>$ is $\theta_f$ and finitely generated. 
Hence, by \cref{fg} $(A,+)$ is finitely generated. 
More precisely, by \cref{gens}, 
the group $(A,+)$ is generated by 
\[
\bigcup_{i=1}^n\left\lbrack t_i\right\rbrack_{\theta} \cup [x]_{\theta},
\]
which contains at most $n'(n+1)$ elements, 
where $n'=\max\{\max\{|[t_i]_\theta|:i\in \{1,...,n\}\},|[x]_\theta|\}$.
So, by \cref{thetafg},
$\Soc(A)$ is of finite index $m=|A:\Soc(A)|\leq n^{n'(n+1)}$ in $A$.
Thus, $\lambda_{mx}(t_i) = t_i$ and $\lambda_{mx}(x) =x$ and $mx \in Z(A,+)$ for $m\leq n^{n'(n+1)}$.
Since $(B,+)$ is an $FC$-group,
$(B,+)/Z(B,+)$ is periodic.
So  there exists $h\in \Z$ such that $hmx\in Z(B,+)$.
Moreover, $mx\in\Soc(A)$, therefore 
$hmx=(mx)^{\circ h}$ and so 
$\lambda_{hmx}=(\lambda_{mx})^h$
is the identity on $A$.
Therefore $hmx\in \Soc(A)\cap Z(B,+)$.
Since $\lambda_{hmx}$ is of finite order,
there exists a $l$ such that $\lambda_{hmx}^l=\textup{id}_B$.
As $hmx\in \Soc(A)$, we see that $\lambda_{lhmx}=\lambda_{hmx}^l=\id_B$.
As $lhmx \in Z(B,+)$, it follows that $lhmx \in \Soc(B)$, which shows the first half of the result. 

Assume now that $B/\Soc(B)$ is periodic. As $B/\Soc(B)$ is an $\theta_f$-skew brace, \cref{Dietzmann cor} shows that every element of $B/\Soc(B)$ has finite multiplicative order. 
Hence, for every $b \in B$ there exists a $n\in \N$ such that $\lambda_b^n=\id_B$, which shows the result.

In the case of $B$ being $\theta_{Bf}$,
let $N\in \N$ such that $[x]_\theta\leq N$.
Going through the calculations before,
we can see that $n,n'\leq N$, $m\leq N^{N(N+1)}$.
Moreover, since $[x]_+\subseteq [x]_\theta$ for every $x\in B$,
using \cref{BFC -> G/Z has finite exp}, we get that $h\leq N!$.
Finally, by \cref{order of lambda for theta Bf}, $ l\leq N!$.
Therefore we can conclude that $N! N!  N^{N(N+1)} x\in\Soc(B)$
for every $x\in B$.
Hence $B/\Soc(B)$ has finite exponent.
\end{proof}

Note that the condition that all $\lambda_b$ are of finite order is natural. If we specialize the result to group theory, by considering almost trivial skew braces, we see that the $\lambda$ map coincides with (additive) conjugation, which has finite order due to the demand that $(B,+)$ is $FC$. Moreover, it is necessary and 
does not follow from $B$ being a $\theta_f$-skew brace, as the following example illustrates.

\begin{exa}\label{notfg}
    Consider the trivial skew braces (of abelian type)
    $$A= \bigoplus_{p \textnormal{ prime}}\mathbb{Z}_p \textnormal{ and }C= \mathbb{Z}.$$
    Then, to form the semidirect product of trivial braces (of abelian type),
    it is sufficient to select a group morphism $ \lambda\colon \mathbb{Z} \rightarrow \Aut(A,+)$.
    Denote by $\varphi_p\colon \mathbb{Z}_p \rightarrow \mathbb{Z}_p$
    a non-identity automorphism of $\mathbb{Z}_p$. 
    Then, set $\varphi = \Pi_{p \textnormal{ prime}}\varphi_p$. 
    Then, the morphism $\lambda(k)=\varphi^k$ is such a group morphism. 
    Hence, we have constructed a skew brace (of abelian type) $B$
    with additive group $(\displaystyle\bigoplus_{p \textnormal{ prime}}\mathbb{Z}_p)\times \mathbb{Z}$
    and multiplicative group $(\displaystyle\bigoplus_{p \textnormal{ prime}}\mathbb{Z}_p) \rtimes_{\varphi} \mathbb{Z}$.
    Clearly, this skew brace is a $\theta_f$-skew brace
    and $\Soc(B)=A\times \ker\lambda$.
    
    However, if $\varphi_p$  has order $p-1$
    for every prime $p$, then
    $\Soc(B)=A\times\{0\}$,
    which means that the quotient $B/\Soc(B) \cong \mathbb{Z}$, which is not periodic.
\end{exa}
The following lemma is technical and probably well-known, but to be self-contained we add a proof.

\begin{lem}\label{oversight}
    A periodic subgroup of an automorphism group of a finitely generated abelian group is finite.
\end{lem}
\begin{proof}
Let $A$ be a finitely generated abelian group.
Then there exists a finite abelian group $G$ and a positive integer $k$ 
such that $ A \cong G \times \Z^k$. 
We claim that $\Aut(A) \cong \Hom(\mathbb{Z}^k,G)\rtimes(\Aut(G) \times GL_k(\Z))$. 
Let $f\in \Hom(\mathbb{Z}^k,G)$ and define the automorphism $\tilde{f}: A \rightarrow A$
with $ \tilde{f}(a,b)=(a+f(b),b)$. 
Then, it can be readily verified that the set 
$H= \left\lbrace \tilde{f} \mid f \in \Hom(\mathbb{Z}^k,G)\right\rbrace$ is a normal subgroup of $\Aut(A)$. 
Consider $T = \left\lbrace \varphi \times \id_{\Z^k} \mid \varphi \in \Aut(G)\right\rbrace$ and $M = \left\lbrace \id_G \times \psi \mid \psi \in GL_k(\mathbb{Z})\right\rbrace$. 
Clearly, $T$ and $M$ commute. 
Hence, as $\Aut(A) = HTM$ and the intersection $H\cap TM = \id$, the claim follows. 

Let $P$ be a periodic subgroup of $\Aut(A)$. Note that both $T$ and $H$ are finite. Hence, $P\cap TH$ is finite. Moreover, 
$$P/(P\cap TH)\cong P+TH/TH \cong \Aut(A)/TH\cong GL_k(\mathbb{Z}),$$
which shows that $P/(P\cap TH)$ is a periodic subgroup of $GL_k(\mathbb{Z})$.
Hence, by a result of Schur (see \cite{robinson2}, 8.1.11), $P/ (P\cap TH)$ is finite, which proves that $P$ is finite.
\end{proof}

The following theorem gives a precise characterization of the skew braces that are $\theta_f$ and where every $\lambda_b$ has finite order. This is an extension of Theorem $2$ by Baer in \cite{FC} characterizing $FC$-groups.
\begin{thm}\label{satisfacttheorem}
    Let $B$ be a skew brace. Then, $B$ is a $\theta_f$-skew brace and every $\lambda_b$ has finite order if and only if the following conditions hold 

    \begin{itemize}
        \item every element of $B$ is contained in a strong left ideal that is finitely generated as a skew brace,
        \item every element in $B/\Soc(B)$ is contained in a finite strong left ideal,
    \end{itemize}
\end{thm}
\begin{proof}
    First, if $B$ is $\theta_f$ and every $\lambda_b$ is of finite order, then the strong left ideal generated by an element $x \in B$ coincides with the subgroup generated by the finite set $[x]_{\theta}$, which shows the first condition. Second, by \cref{unsatisfacttheorem}, it follows that $x$ is a torsion element. Hence, \cref{dietz} guarantees the second condition.

    Vice versa, assume $B$ is a skew brace where every element is contained in a finitely generated strong left ideal and every element of $B/\Soc(B)$ is contained in a finite strong left ideal. We first prove that $\Soc(B) \subseteq \theta_f(B)$. Let $y \in \Soc(B)$. Consider a finitely generated strong left ideal $L$ of $B$ containing $y$. Then, $L+\Soc(B)/\Soc(B)$ is a finitely generated strong left ideal of $B/\Soc(B)$. Hence, by the second condition, $$ |L: L \cap \Soc(B)| =| L+\Soc(B) : \Soc(B) | <\infty.$$ As $ L \cap \Soc(B)\subseteq \Soc(L)$, it follows by \cref{oversocfinite} that $L$ is $\theta_f$. Hence, by \cref{fgtheta}, the additive group $(L,+)$ is finitely generated. Hence, as $L\cap \Soc(B)$ is of finite index, its additive group is also finitely generated. Hence, $y$ is contained in a strong left ideal $I=L\cap \Soc(B)$ contained in $\Soc(B)$ and is finitely generated as a skew brace. As $I$ is contained in $\Soc(B)$ it follows that $I$ is an ideal of $B$ and $(I,+)$ is a finitely generated abelian group. The action of $(B,\circ)$ via $\lambda$ on $I$ induces a group morphism $$ \psi : (B,\circ) \rightarrow \Aut(I,+),$$ where $\Soc(B) \subseteq \ker \psi$. Hence, the image of $\psi$, by the second condition, is a periodic subgroup of $\Aut(I,+)$. As the latter is the automorphism group of a finitely generated abelian group, the image of $\psi$ is a finite group, by \cref{oversight}. In particular, this implies that the action of $(B,\circ)$ on $I$ has finite orbits, which shows that $y \in \lambda_f(B)$, which in turn shows that $\Soc(B) \subseteq \lambda_f(B)$.
    
    Let $x \in B$. Then, $x$ is contained in a finitely generated strong left ideal $L$. Hence, $\Soc(B) + L/\Soc(B)$ is a finitely generated strong left ideal. By the second condition, $\Soc(B) +L / \Soc(B) \cong L /L\cap \Soc(B)$ is a finite strong left ideal. Hence, by \cref{lemma4}, it follows that $L \subseteq \theta_f(B)$. As $x \in L$, this shows $x \in \theta_f(B)$, showing that $B$ is $\theta_f$. Moreover, as $B / \Soc(B)$ is periodic, it follows by \cref{unsatisfacttheorem} that $\lambda_x$ is of finite order. This completes the proof.
\end{proof}

\begin{rem}
    It holds that if $(B,+)$ is $FC$ and all $\lambda$'s are of finite order, then $(B/\Soc(B),\circ)$ is periodic. Indeed, consider any $b \in B$. Then, there exists a positive integer $n$ such that $b^n \in \ker \lambda$. Hence, as $(B,+)$ is $FC$, it follows that there exists a positive integer $k$ such that $k (b^n) \in Z(B,+)$. As $b^n \in \ker \lambda$, one finds that $kb^n = b^{kn} \in Z(B,+) \cap \ker \lambda=\Soc(B)$, which shows the claim.
\end{rem}
\begin{rem}
    In \cite{neumann}, among the most widely known results on $FC$-groups, Neumann proved that the derived subgroup of an $FC$-group is periodic. Another well-known result on this topic is Schur's Theorem (see \cite{robinson2}, 10.1.4). One may naturally expect a similar behavior to arise for $\theta_f$-skew braces and $B^2$, but unfortunately, this does not hold: in \cref{es1} we provide an example of a $\theta_f$-skew brace in which $B^2$ is not periodic and $B/\Soc B$ is finite.
\end{rem}

\section{Applications to solutions}\label{secsolution}

 In this final section, we examine decompositions of solutions. We define the class of $\theta_f$ solutions, which can be understood as unions of finite subsolutions. Moreover, we introduce the $\theta_f$-center $\Delta_f(X,r)$ of a solution $(X,r)$. We show that $\theta_f$ solutions, up to injectivization, correspond to the fact that the structure skew brace is a $\theta_f$-skew brace and vice versa. Moreover, we establish that there is a natural factorization of any structure skew brace $B(X,r)$ into a strong left ideal generated by $\Delta_f(X,r)$ that is a $\theta_f$-skew brace and a strong left ideal. Finally, surprisingly, we show that for involutive solutions the strong left ideal generated by $\Delta_f(X,r)$ coincides with $\theta_f(B(X,r))$.

First, we introduce $\theta_f$-elements of a solution.
\begin{defn}
    Let $(X,r)$ be a solution of the Yang-Baxter equation. Then, we say that an element $x \in X$ is $\theta_f$, if there exists a decomposition $X=Y \cup Z$, where $Y, Z$ are subsolutions such that $x \in Y$ and $Y$ is finite. We say that a solution $(X,r)$ is $\theta_f$, if all the elements of $X$ are $\theta_f$.
\end{defn}

Hence, there is a natural notion of a $\theta_f$-center of a solution, which turns out to be well-behaved.

\begin{pro}\label{thetafofsolution}
    Let $(X,r)$ be a solution of the Yang-Baxter equation. Then, the elements of $X$ that are $\theta_f$ form a decomposition factor of $X$, i.e. these elements form a subsolution $Y$ such that $X\setminus Y$ is also a subsolution. We denote this $\Delta_f(X)$.
\end{pro}
\begin{proof}
    Let $X=Y_i \cup Z_i$ for $i \in I$ be a family of decompositions of $X$. Then, $$ \bigcup_{i \in I} Y_i$$ is a decomposition factor of $X$. Hence, take for any $x \in X$ that is $\theta_f$ a corresponding decomposition $X=Y_x \cup Z_x $ with $x \in Y_x$ and $Y_x$ finite. Then, all elements of $Y_x$ are $\theta_f$. Denote $T$ the set of elements in $X$ that are $\theta_f$. Then, $Y_x \subseteq T$ for all $x \in T$. Hence, $$ T = \bigcup_{x \in T} \left\lbrace x \right\rbrace \subseteq \bigcup_{x \in T} Y_x \subseteq T.$$ Hence, $T$ is the decomposition factor $\displaystyle{\bigcup_{x\in T} Y_x}$, which shows the result. 
\end{proof}

\begin{rem}\label{surjectivedeltaf}
   Note that if $\varphi \colon (X,r) \rightarrow (Y,s)$ is a surjective morphism of solutions, then $\varphi(\Delta_f(X)) \subseteq \Delta_f(Y)$. 
\end{rem}

The following corollary shows that elements that are $\theta_f$ will behave closely like elements of a finite solution, as they are contained in a finite decomposition factor. In particular, the solution $r$ has finite order on any tuple of $\theta_f$-elements. This is an extension of the result obtained in \cite{s}, where it was obtained in the context of $(S)$-skew braces.
\begin{cor}
    Let $(X,r)$ be a solution. Then, any finite number of $\theta_f$-elements are contained in a finite decomposition factor. In particular, if $a,b \in X$ are $\theta_f$, then there exists a positive integer $n$ such that $r^{2n}(a,b)=(a,b)$. 
\end{cor}
\begin{proof}
    Let $x_1,...,x_n \in X$ be $\theta_f$. Then, for each $1\leq i \leq n$ one has that $x_i \in Y_{x_i}$ with $Y_{x_i}$ a finite decomposition factor. Hence, the decomposition factor $\displaystyle{\bigcup_{i=1}^n} Y_{x_i}$ is finite and contains $x_1,...,x_n$.
\end{proof}

In the next result, we connect $\theta_f$-elements of solutions with $\theta_f$-elements of its associated structure skew brace.

\begin{thm}\label{deltafinthetaf}
    Let $(X,r)$ be a solution. Then, $$ \left< \Delta_f(X,r)\right>_+ \subseteq \theta_f(B(X,r)). $$ Furthermore, denoting $(B(X,r),r_B)$ for the solution associated to the skew brace $B(X,r)$ one has that $\theta_f(B(X,r)) = \Delta_f(B(X,r),r_B).$
\end{thm}
\begin{proof}
    By \cref{surjectivedeltaf}, it is sufficient in order to prove the first claim to consider an injective solution. Hence, consider an $x \in X$. Then, for $a,b \in B(X,r)$ one has that $$\theta_{a,b}(x) = \sigma^{-1}_a\lambda_b(x) \in X.$$ In particular, the $\theta$-class of $x$ as an element of the skew brace $B(X,r)$ coincides with the smallest decomposition factor of $X$ that contains $x$. Thus, if $x \in \Delta_f(X,r)$, then the latter is finite, which shows that $x \in \theta_f(B(X,r))$.

    The second claim clearly follows from the definitions.
\end{proof}
The following corollary follows directly from \cref{deltafinthetaf}.
\begin{cor}\label{thetafsolisthetafskew}
    Let $ (X,r)$ be a solution. If $(X,r)$ is $\theta_f$, then $B(X,r)$ is a $\theta_f$-skew brace. Vice versa, if $B(X,r)$ is a $\theta_f$-skew brace and $(X,r)$ is injective, then $(X,r)$ is $\theta_f$.
\end{cor}
As the canonical epimorphism of $(X,r)$ to its injectivization is a factor of the retract morphism \cite{colazzo2023structure}, one can actually conclude the following.
\begin{rem}
    Let $(X,r)$ be a solution such that $B(X,r)$ is a $\theta_f$-skew brace, then $\textup{Ret}(X,r)$ is $\theta_f$.
\end{rem}
However, one can not extend \cref{thetafsolisthetafskew} further to conclude that the original solution is $\theta_f$, as shown by the following example.
\begin{exa}
    Consider the map $r\colon \mathbb{Z} \times \mathbb{Z} \rightarrow \mathbb{Z} \times \mathbb{Z}$ such that $r(x,y) = (y,x+1)$. Then, in $G(\mathbb{Z},r)$ one finds that $i \circ i = i \circ (i+1),$ which shows that the image of $\mathbb{Z}$ is one element and $B(\mathbb{Z},r) \cong \mathbb{Z}$ as a trivial skew brace. Note that this skew brace is clearly $\theta_f$, but the original solution is not, as it consists of a single $\theta$-class.
\end{exa}
Now, we examine the inclusion obtained in \cref{deltafinthetaf} in more detail and observe that for involutive solutions this is an equality.
\begin{pro}\label{involuttheta}
    Let $(X,r)$ be an involutive solution. Then, $$ \theta_f(B(X,r)) = \left< \Delta_f(X,r)\right>_+.$$
\end{pro}
\begin{proof}
    By \cref{deltafinthetaf}, one only needs to prove the inclusion $$\left< \theta_f(B(X,r)) \subseteq \Delta_f(X,r)\right>_+ .  $$ 

    Let $u \in \theta_f(B(X,r))$. As $(B(X,r),+)$ is free abelian, the element $u$ has a unique decomposition $u=d_1x_1+\dots+d_kx_k$ for some mutually different $x_i \in X$ and $d_i \in \mathbb{Z}\setminus\left\lbrace 0 \right\rbrace.$ Denote the support of $u$ by $$ \textup{supp}(u)=\left\lbrace x_1,\dots ,x_k\right\rbrace.$$
    Hence, we need to show that $\textup{supp}(u)\subseteq \Delta_f(X)$. For any $a \in B(X,r)$ the map $\lambda_a$ is an automorphism of $B(X,r)$. Hence, $$ \textup{supp}(\lambda_a(u)) = \lambda_a(\textup{supp}(u)).$$
    Note that for any $i$ one has that $[x_i]_{\theta} \subseteq \textup{supp}([u]_{\theta})$. Hence, $$ |[x_i]_{\theta}|\leq |\textup{supp}( [u]_{\theta})| \leq |\textup{supp}(u)| |[u]_{\theta}| <\infty.$$ Thus, the result follows.
\end{proof}

Note that \cref{involuttheta} is surprising. Indeed, one can view \cref{involuttheta} as the statement that an element is $\theta_f$ if and only if it is the product of generators that are $\theta_f$. A statement that is inherently false for $FC$-elements in groups. Indeed, consider the elements $ab,b \in D_{\infty} =\left< a,b \mid b^2=1, bab=a^{-1}\right>$, then $a=ab b$ is an $FC$-element, whereas $ab$ and $b$ are not $FC$. The following example shows that \cref{involuttheta} can not be extended to non-involutive solutions.
\begin{exa}
    Consider the set $$X = \left\lbrace a^ib \mid i \in \mathbb{Z}\right\rbrace \subset D_{\infty}=\left<a,b \mid b^2=1,bab=a^{-1}\right>$$ and define the map $r\colon X \times X \rightarrow X \times X$ with $r(g,h)=(h,h^{-1}gh)$. Note that $\Delta_f(X,r)=\emptyset$. However, one can show that in $B(X,r)$ one has the relation $ab \circ b = a^{i+1}b \circ a^ib$ for all $i \in \mathbb{Z}$. Moreover, $[ab\circ b]_{\theta}= \left\lbrace ab \circ b,b \circ ab\right\rbrace, $ which shows that $ab \circ b \in \theta_f(B(X,r))$.
\end{exa}

\section*{Acknowledgments}
The first author is a member of GNSAGA (INdAM).
The second author is supported by Fonds Wetenschappelijk Onderzoek - Vlaanderen, via a PhD Fellowship for fundamental research, grant 11PIO24N. 
All authors are members of AGTA - Advances in Group Theory and Applications. The authors thank the anonymous referee for their careful review and useful comments.

\bibliographystyle{abbrv}
\bibliography{refs}

\end{document}